\documentclass[10pt, english, a4paper, nothms]{amsart}
\usepackage[T1]{fontenc}
\usepackage[english]{babel}
\usepackage{lmodern}
\usepackage{mathtools} %
\usepackage{amssymb}
\usepackage[all]{xy}
\usepackage[hidelinks,hyperindex,pdfencoding=unicode]{hyperref}
\xyoption{2cell}
\UseAllTwocells
\usepackage{enumitem}
\usepackage[margin=3cm]{geometry} %
\def\emailaddrname{{\itshape \quad E-mail}}
\def\urladdrname{{\itshape \quad URL}}

\theoremstyle{plain}
\newtheorem{theorem}{Theorem}[section]
\newtheorem*{theorem*}{Theorem}
\newtheorem{prop}[theorem]{Proposition}
\newtheorem*{prop*}{Proposition}
\newtheorem{lemme}[theorem]{Lemma}
\newtheorem{coro}[theorem]{Corollary}

\theoremstyle{remark}
\newtheorem{remark}[theorem]{Remark}
\newtheorem{example}[theorem]{Example}
\theoremstyle{definition}
\newtheorem{defin}[theorem]{Definition}
\newtheorem*{defin*}{Definition}
\newtheorem{paragr}[theorem]{}
\numberwithin{equation}{theorem}

\renewcommand\leq\leqslant
\renewcommand\geq\geqslant

\newcommand\Z{\mathbb{Z}}
\newcommand\N{\mathbb{N}}
\newcommand\G{\mathbb{G}}
\newcommand\Grefl{\mathbb{G}_r}

\newcommand{\Set}{{\mathcal{S}\mspace{-2.mu}\it{et}}}

\newcommand{\Cat}{{\mathcal{C}\mspace{-2.mu}\it{at}}}

\newcommand{\Ch}{{\mathcal{C}\mspace{-2.mu}\it{h}_{+}}}

\let\limind\varinjlim
\let\limproj\varprojlim
\newcommand{\Ob}{\operatorname{\mathsf{Ob}}}
\newcommand{\Arr}{\operatorname{\mathsf{Arr}}}
\newcommand{\Hom}{\operatorname{\mathsf{Hom}}}
\newcommand{\Homi}{\operatorname{\kern.5truept\underline{\kern-.5truept\mathsf{Hom}\kern-.5truept}\kern1truept}}
\newcommand{\id}{\mathrm{id}}

\DeclareMathOperator{\op}{\mathsf{op}}
\newcommand{\pref}[1]{{\widehat{ #1 }}}
\newcommand{\tranche}[2]{#1/{#2}}

\DeclareMathOperator{\ev}{ev}

\newcommand{\W}{\mathcal{W}}
\newcommand{\M}{\mathcal{M}}

\makeatletter
\newcommand{\hocolim@}[2]{%
  \vtop{\m@th\ialign{##\cr
    \hfil$#1\operator@font holim$\hfil\cr
    \noalign{\nointerlineskip\kern1.5\ex@}#2\cr
    \noalign{\nointerlineskip\kern-\ex@}\cr}}%
}
\newcommand{\hocolim}{%
  \mathop{\mathpalette\hocolim@{\rightarrowfill@\textstyle}}\nmlimits@\nolimits
}
\newcommand{\holim@}[2]{%
  \vtop{\m@th\ialign{##\cr
    \hfil$#1\operator@font holim$\hfil\cr
    \noalign{\nointerlineskip\kern1.5\ex@}#2\cr
    \noalign{\nointerlineskip\kern-\ex@}\cr}}%
}
\newcommand{\holim}{%
  \mathop{\mathpalette\holim@{\leftarrowfill@\textstyle}}\nmlimits@
}
\makeatother

\newcommand{\Hotab}{\operatorname{\mathsf{Hotab}}}

\newcommand{\Hot}{\operatorname{\mathsf{Hot}}}

\newcommand{\orient}{\mathcal{O}}

\newcommand{\Ab}{{\mathsf{Ab}}}

\newcommand{\ab}{{\mathsf{ab}}}

\newcommand{\Tot}{\operatorname{\mathsf{Tot}}}

\DeclareMathOperator{\Add}{\mathrm{Add}}
\newcommand{\Addinf}{\operatorname{\mathrm{Addinf}}}
\newcommand{\Homadd}{\operatorname{\kern.5truept\underline{\kern-.5truept\mathsf{Hom}\kern-.5truept}\kern1truept_{\mathrm{add}}}}
\newcommand{\Homaddinf}{\operatorname{\kern.5truept\underline{\kern-.5truept\mathsf{Hom}\kern-.5truept}\kern1truept_{\mathrm{addinf}}}}
\newcommand{\Wh}[2]{\Z_{#1}^{\left(#2\right)}}
\newcommand{\Whf}[1]{\operatorname{\mathsf{Wh}_{#1}}}

\renewcommand{\H}[2]{\operatorname{{H}}(#1,#2)}

\newcommand{\Hf}[1]{\operatorname{H_{#1}}}

\newcommand{\prefab}[1]{\widehat{ #1 }^{\mathsf{ab}}}

\newcommand{\Wab}{\W^{\ab}}

\DeclareMathOperator{\im}{Im}
\DeclareMathOperator{\U}{U}
\DeclareMathOperator{\dk}{\mathsf{C}_N}

\DeclareMathOperator{\wcat}{\omega\Cat}

\newcommand{\cub}{\square}

\newcommand\EnsSimp{\widehat{\Delta}}

\newcommand{\nerf}{\mathsf{N}}

\let\epsilon\varepsilon

\newcommand{\mdvirg}{\text{ ,}\hspace{10pt}}
\newcommand\zbox[1]{\makebox[0pt][l]{#1}}
\newcommand\pbox[1]{\zbox{\quad #1}} %

\newcommand{\overbar}[1]{\mkern 1.5mu\overline{\mkern-1.5mu#1\mkern-1.5mu}\mkern 1.5mu}
\renewcommand{\tilde}{\widetilde}

\title[A Homotopical Dold-Kan correspondence for $\Theta$ and other test
categories]{%
A Homotopical Dold-Kan correspondence for Joyal's category $\Theta$ and other test categories
}
\author{Léo Hubert}
\address{Léo Hubert, Aix-Marseille Univ,
         CNRS, I2M,
         Marseille,
         France}
\email{leo.hubert.math@mailo.com}
\urladdr{https://leohubert.pages.math.cnrs.fr}
\begin{document}

\begin{abstract}
We prove that for any test category $A$, in the sense of Grothendieck,
satisfying a compatibility condition between homology equivalences and weak
equivalences of presheaves, the homotopy category of abelian presheaves on $A$
is equivalent to the non-negative derived category of abelian groups. This
provides a homotopical generalization of the Dold-Kan correspondence for
presheaves of abelian groups over a wide range of test categories.
This equivalence of homotopy categories comes from a Quillen equivalence for a
model structure on abelian presheaves that we introduce under these conditions.
We then show that this result applies to Joyal's category $\Theta$.
\end{abstract}

\maketitle

\section*{Introduction}
In the first chapters of his $1983$ manuscript \emph{Pursuing Stacks}
\cite{pursuingstacks}, Grothendieck introduced the notion of \emph{test
categories} to get an axiomatic characterization of small categories whose
category of presheaves modelizes, in a canonical way, the category~$\Hot$ of
topological spaces up to weak homotopy equivalences. This theory has since been
described and developed further, notably
in~\cite{maltsiniotis2005},~\cite{cisinskipref},~\cite{jardine2006categorical},~\cite{cisinski2011theta},
\cite{ara2019DendroidalTest} and \cite{ara2012thetatest}. 

However, Chapters V to VII of \emph{Pursuing Stacks} have not met the same
enthusiasm yet. In these chapters, Grothendieck introduced a formalism to study
the homology theory of abelian presheaves on small categories, different from
the one used in functor homology (see for example
\cite{pirashvili2003introduction}). He also conjectured a homotopical
generalization of the Dold-Kan correspondence for abelian presheaves on test
categories. The present article aims to expose the basics of this formalism and
to prove Grothendieck's conjecture, under additional conditions applying notably
to Joyal's category $\Theta$ \cite{joyal1997disks}.

\subsection*{A brief overview of test categories}
In this context, the category $\Hot$ is realized, using Illusie-Quillen's
theorem \cite{illusie1971cotangent}, as the Gabriel-Zisman localization of the
category $\Cat$ of small categories by the class~$\W_\infty$ of functors whose
image by the nerve functor is a simplicial weak equivalence. Given any small
category $A$, the functor
\[
  i_A : \pref{A} \to \Cat \mdvirg X \mapsto \tranche{A}{X}
\]
sending every presheaf to its category of elements allows to define a class of
weak equivalences of presheaves denoted by $\W_A$, consisting of the natural
transformations in $\pref{A}$ whose image by the functor $i_A$ is in
$\W_\infty$. Grothendieck denotes by $\Hot_A$ the localization of~$\pref{A}$ by
$\W_A$, and defines a \emph{pseudo-test category} to be a category $A$ such that
the functor $\Hot_A \to \Hot$ induced by $i_A$ is an equivalence of categories.

However, this notion is too general for Grothendieck's purpose of finding
elegant characterizations, and he gradually restricts his scope. First, he
defines \emph{weak test categories} by asking, in addition, for the right
adjoint $i_A^* : \Cat \to \pref{A}$ of $i_A$ to preserve weak equivalences. 
He then manages to find very simple characterizations of categories~$A$ whose
slices $\tranche{A}{a}$ are weak test categories for every object $a$ of $A$,
that he calls \emph{local test categories}. Moreover, any local test category
with trivial homotopy groups is also a weak test category: \emph{test
categories} are then defined to be categories which are both local and weak test
categories. Finally, \emph{strict test categories} are test categories $A$ such
that the functor $i_A$ commutes, up to weak equivalences, with the product of
presheaves.

Important examples of test categories include the category $\Delta$ of non-empty
finite ordinals, the cubical category $\cub$ and its variations with connections
\cite{maltsiniotis2009cubique}, and the dendroidal category
$\Omega$~\cite{ara2019DendroidalTest}. A particularly important example is given
by the category $\Theta$ who, besides being a strict test
category~\cite{cisinski2011theta}, is deeply linked to both strict
$\omega$-categories~(\cite{makkai2001duality},\cite{berger2002cellular}) and
weak higher
categories~(\cite{joyal1997disks},\cite{batanin1998monoidal},\cite{rezk2010weak},\cite{ara2014quasi}).

\subsection*{Linearization of homotopy types}
The homology of a small category is traditionally defined as the homology of its
simplicial nerve, and we can therefore define the homology of a presheaf on any
small category as the homology of its category of elements. However, this
process is rather complex, and gives a preeminent role to the category~$\Delta$. 
In the context of test categories, Grothendieck asks in \cite{pursuingstacks} a
handful of questions regarding what he calls the \emph{linearization of homotopy
types}.

Denote by $\Hotab$ the non-negative derived category of abelian groups, and
by~$\prefab{A}$ the category of abelian presheaves on $A$.
\begin{enumerate}
\item Given a small category $A$, can we define a canonical functor
\[
  \Hf{A} : \prefab{A}\to\Hotab
\]
  so that the following diagram 
\[
\xymatrix{
\pref{A} \ar[d]_-{\Z} \ar[r]^-{i_A} & \Cat \ar[r]^-{\nerf} &
\pref{\Delta} \ar[r]^-{\dk} & \Ch(\Ab) \ar[d]^{can} \\
\prefab{A} \ar[rrr]_{\Hf{A}} &&& \Hotab \pbox{,}
} 
\]
where $\dk$ is the normalized chain complex functor and $\nerf$ is the
simplicial nerve, is commutative up to a natural isomorphism? 
\end{enumerate}
We introduce in this paper such a functor, defined, following Grothendieck's
ideas, as the left derived functor of the colimit functor. 

This allows to define a class of weak equivalences of abelian presheaves on $A$
that we denote by $\Wab_A$, whose elements are the morphisms sent to
isomorphisms by the functor $\Hf{A}$. Note that, in the case of simplicial
abelian groups, one usually takes as weak equivalences morphisms whose
underlying morphism of simplicial sets is a weak equivalence. A very important
observation in our context is the fact that the functor $\dk : \prefab{\Delta}
\to \Ch(\Ab)$ sends these morphisms to quasi-isomorphisms.

\begin{enumerate}[resume]
\item Let $\U : \prefab{A}\to\pref{A}$ be the forgetful functor.
Can we exhibit conditions on $A$ to ensure that the functor $\Hf{A}$ sends
morphisms of $\U^{-1}\W_A$ to isomorphisms in $\Hotab$? 
\end{enumerate}
We will focus on categories, such as $\Delta$, for which the classes $\Wab_A$
and $\U^{-1}\W_A$ coincide. We call this the \emph{strong Whitehead condition}.
We won't study in this paper categories satisfying only the
inclusion~$\U^{-1}\W_A \subset \Wab_A$.

Finally, the functor $\dk : \prefab{\Delta}\to\Ch(\Ab)$ has another important
property, expressed by the Dold-Kan theorem: it is an equivalence of categories
(\cite{dold1958homology},\cite{kan1958functors}). 
In general, abelian presheaves on test categories are not equivalent to
non-negative chain complexes of abelian groups. 
For example, the category of bisimplicial abelian groups is equivalent to the
category of double chain complexes of the first quadrant.
However, a reasonable generalization of the Dold-Kan correspondence would be an
answer to the following question:
\begin{enumerate}[resume]
\item If the answer to question $(2)$ is positive, we get a functor 
\[
\overline{\Hf{A}} : \Hotab_A^{\U} \to \Hotab
\]
where $\Hotab_A^{\U}$ is the localization of $\prefab{A}$ by the class
$\U^{-1}\W_A$.
Is it true that if $A$ is a test category, then the functor $\Hf{A}$ is an
equivalence of categories? 
\end{enumerate}

The present article shows that the answer to this last question is positive for
categories satisfying the \emph{strong} Whitehead condition.

\begin{theorem*}[\ref{thm:WhiteheadTestHomPseudoTest}]
Let $A$ be a test category satisfying the strong Whitehead condition. 
Then the functor $\overline{\Hf{A}} : \Hotab_A^{\U} \to \Hotab$ is an
equivalence of categories. 
\end{theorem*}

A major ingredient of the proof comes from functors satisfying the condition of
Quillen's Theorem~A \cite[Theorem A]{quillenktheory}, that are called
\emph{aspherical functors} in the context of Grothendieck's homotopy theory.
Given any aspherical functor $u : A \to B$, the induced restriction functor $u^*
: \pref{B}\to\pref{A}$ is homotopy cofinal, meaning that it induces an
isomorphism between homotopy colimits.

We show that if $u : A \to B$ is an aspherical functor and $A$ satisfies the
strong Whitehead condition, then the same holds for $B$. We also study the class
$\Wab_\infty$ of functors between small categories which induce isomorphisms in
homology, and the corresponding notion of $\Wab_\infty$-aspherical morphisms.
We show that the restriction functor $(u^*)^\ab : \prefab{B}\to\prefab{A}$
induced by a $\Wab_\infty$-aspherical functor commutes with the homology
functor. Moreover, using the Grothendieck-Cisinski model structure on the
category of presheaves over a local test category
\cite[Corollary~4.2.18]{cisinskipref}, we prove the following theorem:

\begin{theorem*}[]
Let $A$ be a local test category satisfying the strong Whitehead condition.
\begin{itemize}
\item[\normalfont{(\ref{thm:modelStructureAbPresheaves})}]
The category $\prefab{A}$ can be given the structure of a cofibrantly generated
model category whose weak equivalences are the elements of the class $\Wab_A =
\U^{-1}\W_A$. 
\item[\normalfont{(\ref{prop:foncteurAspheriqueEqQuillenAb})}]
If $u : A \to B$ is an aspherical functor and $B$ is also a local test category,
then the functor~$(u^*)^\ab : \prefab{B} \to \prefab{A}$ is a left Quillen
equivalence.
\end{itemize}
\end{theorem*}

In particular, the equivalence of localized categories of Theorem
\ref{thm:WhiteheadTestHomPseudoTest} can be promoted to a Quillen equivalence
with the projective model structure on non-negative chain complexes of abelian
groups.

As an application of Theorem \ref{thm:WhiteheadTestHomPseudoTest}, we prove that
Joyal's category $\Theta$ satisfies the homotopical Dold-Kan correspondence. 
Our contribution to this proof consists mainly in the study of a variation on
the notion of aspherical functors for functors whose codomain is a presheaf
category. We prove that if $A$ satisfies the strong Whitehead condition, and if
$j: B \to \pref{A}$ is an aspherical functor in this sense, then $B$ also
satisfies the strong Whitehead condition. We then use previous work from Ara and
Maltsiniotis \cite{ara2022compnerfs} showing that the functor $\Delta \to
\pref{\Theta}$ given by the cellular nerve (\cite{batanin2000multitude},
\cite{berger2002cellular}) and Street's oriental functor
\cite{street1987algebra} is an aspherical functor. This then implies that
$\Theta$ satisfies the strong Whitehead condition, and therefore satisfies the
conditions of Theorem \ref{thm:WhiteheadTestHomPseudoTest}. 

We would like to point out that one can also try to characterize small
categories $A$ for which the functor $\Hf{A}$ induces an equivalence of
categories between the localization of the pair $(\prefab{A},\Wab_A)$ and
$\Hotab$, without requiring the condition of question $(2)$. This is partly done
in the author's PhD thesis \cite{hubert2025phd} from which this article is
extracted, with still more investigations to be held.
For example, the category $\tranche{\Delta}{\nerf(\Grefl)}$, where $\Grefl$ is
the reflexive globular category, is a test category that doesn't satisfy the
strong Whitehead condition. However, one can prove using
\cite[Proposition~4.3.6]{hubert2025phd} that it satisfies this version of the
homotopical Dold-Kan correspondence.

\subsection*{Structure of the paper} 
Section \ref{sec:testCategories} consists of a brief recollection of definitions
and results concerning test categories that we will use throughout this paper.
In Section \ref{sec:homology}, we introduce the homology functor and show that,
when restricted to presheaves of sets, it computes the homology of the category
of elements. In Section \ref{sec:integrators}, we expose Grothendieck's
definition of the tensor product of functors and use it to produce explicit
computations of the homology functor. In Section \ref{sec:asphericalMorphisms},
we introduce the class $\Wab_\infty$ of functors inducing isomorphisms in
homology and we prove that the restriction functor induced by a
$\Wab_\infty$-aspherical morphism commutes with homology. Section
\ref{sec:whitehead} introduces the strong Whitehead condition, and contains the
two main theorems stated in the introduction. In Section \ref{sec:theta}, we
focus on Joyal's category $\Theta$ and show, using Ara and Maltsiniotis's
results \cite{ara2022compnerfs}, that $\Theta_n$, for any integer $n\geq1$, and
$\Theta$ are strong Whitehead categories, and thus verify the homotopical
Dold-Kan correspondence.

\subsection*{Acknowledgment}
I am grateful to Georges Maltsiniotis, whose project to edit \emph{Pursuing
Stacks} made the present work possible, and to Dimitri Ara for his guidance
during my PhD, from which this article is extracted.

All definitions in Sections \ref{sec:testCategories} to \ref{sec:integrators},
except for the term \emph{homologically pseudo-test}, are due to Grothendieck
\cite{pursuingstacks}. 

\subsection*{Notation}
If $\mathcal{C}$ is a category, we denote by $\Ob(\mathcal{C})$ its class of
objects and by $\Arr(\mathcal{C})$ its class of morphisms. Given another
category $\mathcal{D}$, we denote by $\Homi(\mathcal{C},\mathcal{D})$ the
category of functors $\mathcal{C}\to\mathcal{D}$ and natural transformations
between them. We denote by $\Set$ the category of sets, $\Ab$ the category of
abelian groups and $\Cat$ the category of small categories. 

If $A$ is a small category, $\pref{A}$ is the category $\Homi({A}^{\op}, \Set)$
of set-valued presheaves on $A$. Given an object $a$ of $A$, we denote by $a$
the associated representable presheaf on $A$. Given a functor~$u: A \to B$
between small categories, we denote by~$u^* : \pref{B} \to \pref{A}$ the
restriction functor, and by $u_!$ (resp.~$u_*$) its left (resp.~right) adjoint.

We denote by $e$ the final object of $\Cat$, and by $e_{\pref{A}}$ the final
object of $\pref{A}$ for any small category~$A$.

\section{Preliminaries on test categories}
\label{sec:testCategories}

We begin by a brief recollection of definitions and facts concerning the theory
of test categories that will be used throughout this paper. We refer to
\cite{maltsiniotis2005} for a detailed exposition of this theory.

\begin{paragr}
If $u: A \to B$ is a functor and $b$ is an object of $B$, we denote by
$\tranche{A}{b}$ the category whose objects are pairs~$(a,f)$ where $a$ is an
object of $A$ and $f: ua \to b$ is a morphism in $B$, and whose morphisms from
an object $(a,f)$ to an object $(a',f')$ are morphisms $g: a \to a'$ such that
$f'u(g)=f$. In particular, if $A$ is a small category, $u: A \to \pref{A}$ is
the Yoneda embedding and $X$ is a presheaf on $A$, the category
$\tranche{A}{X}$ is the category of elements of $X$. This defines a
functor 
\begin{align*}
i_A: \pref{A} \to \Cat \mdvirg A \mapsto \tranche{A}{X}
\end{align*}
which admits as a right adjoint the functor 
\[
i_A^*: \Cat \to \pref{A} \mdvirg C \mapsto \Hom(\tranche{A}{a}, C) \pbox{.}
\]
\end{paragr}

\begin{paragr}
A class $\W \subset \Arr(\Cat)$ is \emph{weakly saturated} if it
satisfies the following conditions: 
\begin{enumerate}
\item identities are in $\W$;
\item if two out of three arrows in a commutative triangle in $\Cat$ are in
$\W$, then so is the third;
\item if $i: X' \to X$ and $r: X \to X'$ are two functors between small
categories such that $ri=\id_{X'}$ and~$ir$ is in $\W$, then $r$ is also in
$\W$;
\end{enumerate}
\end{paragr}
\begin{paragr}
A class $\W \subset \Arr(\Cat)$ is a \emph{basic localizer} if it is weakly
saturated and satisfies the following conditions (recall that $e$ is the final
object of $\Cat$): 
\begin{enumerate}[resume]
\item if $A$ is a small category with a final object, then the
morphism $A\to e$ is in $\W$;
\item if $u: A \to B$ is a functor in $\Cat$ such that for every
object $b$ of $B$, the functor 
\[
\tranche{u}{b}: \tranche{A}{b} \to \tranche{B}{b} \mdvirg (a,u(a)\to b)
\mapsto (u(a),u(a)\to b) 
\]
is in $\W$, then $u$ is in $\W$. 
\end{enumerate}
\end{paragr}

\begin{paragr}
\label{par:defAspherique} 
Let $\W$ be a basic localizer in $\Cat$. We say that a small category
is \emph{$\W$-aspherical} if the unique morphism $A \to e$ is in $\W$. 
We say that a functor $u: A \to B$ is $\W$-aspherical if for every
object $b$ of $B$, the morphism $\tranche{u}{b}$ is in $\W$. Finally, we say
that a presheaf $X$ on $A$ is $\W$-aspherical if $i_A(X)$ is
a~$\W$-aspherical category. One can prove that any representable presheaf is $
\W$-aspherical, since for any object $a$ of~$A$, the
category~$\tranche{A}{a}$ has a final object.
\end{paragr}

\begin{prop}[Grothendieck]\label{prop:MorphismesAspheriquesLocFond}
Let $u:A\to B$ be a functor between small categories. The following
conditions are equivalent:
\begin{enumerate}
\item u is $\W$-aspherical;
\item a presheaf $X$ on $B$ is $\W$-aspherical if and only if $u^*(X)$ is
$\W$-aspherical;
\item for every object $b$ of $B$, the presheaf $u^*(b)$ is $\W$-aspherical;
\item for every presheaf $X$ on $B$, the functor
\[
  \tranche{u}{X}: \tranche{A}{X}\to \tranche{B}{X} \mdvirg (a, u(a) \rightarrow X)
  \mapsto (u(a), u(a) \rightarrow X)
  \]
  is in $\W$;
\end{enumerate}

Moreover, these conditions imply the following condition:
\begin{enumerate}[resume]
\item for any morphism $f: X \to Y$ in $\pref{B}$, the functor
$i_B(f)$ is in $\W$ if and only if $i_A(u^*(f))$ is in $\W$.
\end{enumerate}
\end{prop}
\begin{proof}
See \cite[Proposition 1.2.9]{maltsiniotis2005}.
\end{proof}

\begin{paragr}
\label{def:eqThomason} 
We denote by $\W_\infty$ the class of morphisms in $\Cat$ sent to
simplicial weak equivalences by the simplicial nerve functor. Using
\cite[Theorem A]{quillenktheory}, it is easy to prove that $\W_\infty$ is a basic
localizer. We call \emph{Thomason equivalences}, or simply \emph{weak
equivalences of categories}, its elements. 
\end{paragr}
\begin{paragr}
We denote by $\Hot$ the localization of $\Cat$ by the class $\W_\infty$. Recall
(\cite[VI, Theorem 3.3]{illusie1971cotangent}) that the simplicial nerve functor
induces an equivalence of categories between $\Hot$ and the homotopy category of
simplicial sets. 
\end{paragr}

In the rest of this paper, we will simply use the adjective \emph{aspherical}
to refer to $\W_\infty$-asphericity.

\begin{paragr}
Let $A$ be a small category. A morphism $f: X \to Y$ of $\pref{A}$ is a \emph{weak
equivalence} if its image by the functor $i_A$ is in $\W_\infty$. We
denote by $\W_A$ the class of weak equivalences of presheaves on $A$,
and by $\Hot_A$ the localization of $\pref{A}$ by the class $\W_A$.

The functor $i_A$ induces a functor $\Hot_A \to \Hot$, and we say that
$A$ is a \emph{pseudo-test category} if this functor is an equivalence of
categories. 
\end{paragr}

\begin{paragr}
We say that $A$ is a \emph{weak test category} if the following
conditions are satisfied: 
\begin{enumerate}
\item $i_A^*(\W_\infty) \subset \W_A$;
\item for every presheaf $X$, the unit morphism $\eta_X: X \to
i_A^*i_A(X)$ is in $\W_A$;
\item for every small category $X$, the counit morphism $\epsilon_C:
i_Ai_A^*(C)\to C$ is in $\W_\infty$.
\end{enumerate}
In particular, every weak test category is a pseudo-test category.
\end{paragr}

 \begin{paragr}
A small category $A$ is a \emph{local test category} if for any object
$a$ of $A$, the category $\tranche{A}{a}$ is a weak test category. If
$A$ is both a weak test category and a local test category, we say
that $A$ is a \emph{test category}.
 \end{paragr}

\begin{theorem}[Grothendieck-Cisinski]\label{thm:modelStructureCisinski}
Let $A$ be a local test category. The category $\pref{A}$ can be given
the structure of a model category where the weak equivalences are the elements
of $\W_A$ and the cofibrations are the monomorphisms.
\end{theorem}
\begin{proof}
See \cite[Corollary 4.2.18]{cisinskipref}.
\end{proof}

\begin{prop}[Cisinski]\label{StructureCisinskiMorphAspheriquesEqQuillen}
Let $u:A \to B$ be an aspherical functor between local test
categories. Then $(u^*,u_*)$ is a Quillen equivalence for the
Grothendieck-Cisinski model structure on $\pref{A}$ and $\pref{B}$.
\end{prop}
\begin{proof}
See \cite[Proposition 6.4.29]{cisinskipref}.
\end{proof}

\begin{paragr}
Let $A$ be a small category. A presheaf $X$ on $A$ is \emph{locally
$\W$-aspherical} if the canonical functor~$i_A(X)\to A$ is
$\W$-aspherical. We say that $A$ is \emph{totally aspherical} if it is
aspherical, and if for any presheaves $X$ and $Y$ on $A$, the
canonical morphism $i_A(X\times Y) \to i_A(X)\times i_A(Y)$ is
aspherical. 
\end{paragr}

\begin{prop}[Grothendieck]\label{prop:totalementAspheriqueEquivalences}
Let $A$ be a small category. The following conditions are equivalent: 
\begin{enumerate}
\item $A$ is totally aspherical;
\item the diagonal functor $A \to A \times A$ is aspherical;
\item every aspherical presheaf on $A$ is locally aspherical.
\end{enumerate}
\end{prop}
\begin{proof}
See \cite[Proposition 1.6.1]{maltsiniotis2005}.
\end{proof}

\begin{prop}[Grothendieck]\label{totAspheriqueMorphismeAspherique}
If $A$ is a totally aspherical category and $u: A \to B$ is an aspherical
functor, then $B$ is also totally aspherical.
\end{prop}
\begin{proof}
See \cite[Proposition 1.6.5]{maltsiniotis2005}.
\end{proof}

\begin{paragr}
A \emph{strict test category} is a test category which is also totally
aspherical.
\end{paragr}

\begin{prop}[Grothendieck]
Let $A$ be a totally aspherical category. The following conditions are
equivalent: 
\begin{enumerate}
\item $A$ is a weak test category;
\item $A$ is a local test category;
\item $A$ is a test category;
\item $A$ is a strict test category.
\end{enumerate}
\end{prop}
\begin{proof}
See \cite[Proposition 1.6.6]{maltsiniotis2005}.
\end{proof}

\begin{prop}[Grothendieck]\label{propFonctAspheriqueSourceTotAsphTest}
Let $u: A \to B$ be a functor. If $A$ is totally aspherical and if $B$
is a local test category, then $A$ and $B$ are both strict test
categories.
\end{prop}
\begin{proof}
See \cite[Corollary 1.6.11]{maltsiniotis2005}.
\end{proof}

\begin{paragr}
Let $A$ be a small category and $i: A \to \Cat$ a functor. Consider
the functor 
\[
i^*: \Cat \to \pref{A} \mdvirg C \mapsto \Hom_\Cat(i(a),C) \pbox{.}
\]
Observe that if $i$ is the inclusion functor from the simplex category
$\Delta$ to $\Cat$, then $i^*$ is the simplicial nerve functor. We say
that $i$ is an \emph{aspherical functor} if it satisfies the following
conditions: 
\begin{enumerate}
\item for every object $a$ of $A$, the category $i(a)$ is aspherical;
\item a small category $C$ is aspherical if and only if the presheaf
$i^*(C)$ is aspherical.
\end{enumerate}
\end{paragr}

\begin{prop}[Grothendieck]
\label{prop:testFaiblei_AAspherique}
Let $A$ be a small category. Then $A$ is a weak test category if and only if
$i_A$ is an aspherical functor.
\end{prop}
\begin{proof}
See \cite[Remark 1.7.2]{maltsiniotis2005}.
\end{proof}

\begin{prop}[Grothendieck]\label{lemmeFoncteursAspheriquesTriCommutatif}
Let $u: A \to B$ be a functor between small categories, and $j: B \to \Cat$ a
functor such that for every object $b$ of $B$, the category $j(b)$ is
aspherical.
\begin{enumerate}
\item if $u$ is aspherical, then $ju: A \to \Cat$ is aspherical if
and only if $j$ is aspherical;
\item if $j$ is fully faithful and $ju: A \to \Cat$ is an aspherical functor, then
$u$ and $j$ are aspherical.
\end{enumerate}
\end{prop}
\begin{proof}
See \cite[Lemma 1.7.4]{pursuingstacks}.
\end{proof}

\begin{paragr}
Let $A$ be a small category. We say that a functor $i: A \to \Cat$ is
a \emph{weak test functor} if $A$ is a weak test category and if $i$ is
aspherical. We say that $i$ is a \emph{local test functor} if for every
object $a$ of $A$, the functor $\tranche{A}{a}\to\Cat$ induced by $i$
is a weak test functor. Finally, $i$ is a \emph{test functor} if it is
both a weak test functor and a local test functor. 
\end{paragr}

\begin{theorem}[Grothendieck]
\label{thm:EquivalencesFoncteursTestLocal}
Let $A$ be a small category and $i:A \to \Cat$ a functor such that for
every object $a$ of $A$, the category $i(a)$ is aspherical. The
following conditions are equivalent:
\begin{enumerate}
\item $i$ is a local test functor (and $A$ is a local test category);
\item  for every small aspherical category $C$, the presheaf $i^*(C)$
is locally aspherical.
\end{enumerate}
\end{theorem}
\begin{proof}
See \cite[Theorem 1.7.13]{maltsiniotis2005}.
\end{proof}

\section{Homology, homotopy colimits and the category of elements} 
\label{sec:homology}

In the context of test categories, the homology of a presheaf is the homology of
its category of elements. 
We will now introduce, following Grothendieck's ideas in~\cite{pursuingstacks},
a homology functor for \emph{abelian} presheaves that generalizes this idea. We
show that it also generalizes the homology of groups and of simplicial abelian
groups.

\begin{paragr}
If $A$ is a small category, we denote by $\prefab{A} =
\Homi({A}^{\op},\Ab)$ the category of abelian presheaves on $A$. Given a
presheaf $X$ on~$A$, we denote by $\Wh{A}{X}$ the abelian presheaf on $A$
sending every object $a$ of $A$ to the free abelian group on the set $Xa$.
Following Grothendieck's terminology \cite{pursuingstacks}, we will call the
\emph{Whitehead functor} the functor 
\[
\Whf{A}: \pref{A}\to\prefab{A} \mdvirg X \mapsto \Wh{A}{X} \mdvirg
\]
or simply the \emph{abelianization functor}.
If $a$ is an object of $A$, we will simply write $\Wh{A}{a}$ for the image of
the representable presheaf $a$ by the functor $\Whf{A}$, and we will refer to it
as the abelian presheaf represented by $a$.
\end{paragr}

\begin{paragr}\label{paragr:HomologyFunctor}
We denote by $\Hotab$ the derived category $\mathcal{D}_+(\Ab)$, obtained by
localizing non-negative chain complexes of abelian groups by quasi-isomorphisms.
We define a functor 
\[
\Hf{A}: \prefab{A} \to \Hotab \mdvirg X \mapsto \H{A}{X}
\]
where $\H{A}{X}$ is the homotopy colimit of the diagram 
\[
{A}^{\op} \xrightarrow{X} \Ab \xhookrightarrow{in_0} \Ch(\Ab) \pbox{,}
\]
$in_0$ being the inclusion of abelian groups into chain complexes
concentrated in degree $0$.
We say that $\H{A}{X}$ is the \emph{homology of $A$ with coefficients in $X$},
or simply the \emph{homology of~X}.
\end{paragr}

\begin{example}
If $A$ is the groupoid $BG$ associated to a group $G$, abelian presheaves on $A$
coincide with right $\Z G${\nobreakdash}-modules. If $M: {BG}^{\op}\to\Ab$ is
such a $\Z G${\nobreakdash}-module, its colimit as a diagram in $\Ab$ is the
abelian group $M_G$ of coinvariants, and $\H{BG}{M}$ is the complex, up to
quasi-isomorphism, computing the homology of the group $G$ with coefficients in
$M$. 
\end{example}

\begin{example}\label{ex:normalizedComplexHomology}
If $A$ is the simplex category $\Delta$, then the diagram 
\[
\xymatrix{
\prefab{\Delta} \ar[rd]_{\Hf{\Delta}} \ar[r]^-{\dk} & \Ch(\Ab)
\ar[d]^{} \\
& \Hotab
} 
\]
is commutative up to a natural isomorphism, where $\dk$ is the normalized chain
complex functor 
\[
\dk: X \mapsto 
\left(
X_0 \xleftarrow{d_0} \cdots \xleftarrow{d_0}
\bigcap_{0<i\leq n}\ker{d_i}
\xleftarrow{d_0}
\bigcap_{0<i\leq n+1}\ker{d_i}
\leftarrow \cdots
\right)
\pbox{.}
\]
Thus if $X$ is a simplicial abelian group, $\H{\Delta}{X}$ is the complex, up to
quasi-isomorphism, computing the standard homology of $X$. This result already
appears in \cite[XII, 5.5]{bousfield1972homotopy}, but we will prove it again
in Example \ref{ex:IntegratorDelta}. 
Note that we could also use the non-normalized chain complex functor, since both
are pointwise quasi-isomorphic.
\end{example}

We will now prove that the homology functor we just introduced associates to any
presheaf of sets the homology of its category of elements. 

\begin{paragr}
Recall that the category $\pref{\Delta}$ can be equipped with the
Kan\nobreakdash-Quillen model structure \cite[II, Section 3, Theorem
3]{quillenhomotalg}, where weak equivalences are the simplicial weak
equivalences and cofibrations are the monomorphisms.
Similarly, the category~$\Ch(\Ab)$ can be equipped with the projective model
structure \cite[XII,4.12]{quillenhomotalg} where the weak equivalences are the
quasi-isomorphisms, and the cofibrations are the monomorphisms with projective
cokernels. Moreover, the functor 
\[
\pref{\Delta}\xrightarrow{\Whf{\Delta}} \prefab{\Delta} \xrightarrow{\dk} \Ch(\Ab) 
\]
is a left Quillen functor with respect to these model structures, and it
preserves weak equivalences.
\end{paragr}

\begin{prop}[Thomason]\label{i_Ahocolim}
Let $X$ be a presheaf on a small category $A$. Denote by $\widetilde{X}$ the
diagram of discrete simplicial sets 
\[
\widetilde{X}: {A}^{\op} \to \Set \hookrightarrow \Cat
\xrightarrow{\nerf} \EnsSimp 
\]
where $\nerf$ is the simplicial nerve functor. Then there is a natural
isomorphism 
\[
\nerf i_A(X) \simeq \hocolim_{{A}^{\op}}^{\Hot_\Delta} \tilde{X}
\]
in $\Hot_\Delta$.
\end{prop}
\begin{proof}
This is \cite[Theorem 1.2]{thomason1979homotopy}.
\end{proof}

\begin{prop}\label{prop:colimiteHomotopiqueLibres}
For any small category $A$, the following diagram is commutative up to
a natural isomorphism:
\begin{equation}\label{diag:colimitehomotopiquelibres}
\xymatrix@C=3em{
\pref{A} \ar[r]^{i_A} \ar[d]_{\Whf{A}} & \Cat \ar[r]^{\nerf} 
& \Hot_{\Delta} \ar[d]^-{{\dk\Whf{\Delta}}} \\
\prefab{A} \ar[rr]_{\Hf{A}} && \Hotab & \pbox{.}
} 
\end{equation}
\end{prop}
\begin{proof}
If $X$ is a presheaf on $A$, we have the following natural
isomorphisms in $\Hotab$:
\begin{align*}
{\dk\Whf{\Delta}\nerf} i_A(X) 
&\simeq
{\dk\Whf{\Delta}}\hocolim_{{A}^{\op}}^{\Hot_{\Delta}} \tilde{X}
\\
&\simeq \hocolim_{{A}^{\op}}^{\Hotab} \dk \circ \Whf{\Delta} \circ \tilde{X}
\end{align*}
where the first isomorphism comes from Proposition \ref{i_Ahocolim} and the
second from the fact that left Quillen functors commute with homotopy colimits.
To conclude, note that the diagram 
\[
{A}^{\op} \xrightarrow{\tilde{X}}  
\pref{\Delta} \xrightarrow{\Whf{\Delta}} 
\prefab{\Delta} \xrightarrow{\dk} 
\Ch(\Ab)
\]
coincide with the diagram $\Wh{A}{X}[0]$, associating to any object $a$ of $A$
the chain complex concentrated in degree $0$ of value~$\Z^{(Xa)}$. We then get
the natural isomorphism in
$\Hotab$ 
\begin{align*}
\dk\Whf{\Delta}\nerf i_A(X) 
&\simeq \hocolim_{{A}^{\op}}^{\Hotab}(\Wh{A}{X}[0])
\end{align*}
which concludes the proof.
\end{proof}

\begin{coro}\label{coro:homologyCoefRep}
For any small category $A$ and any object $a$ in $A$, there is a
natural isomorphism 
\[
\H{A}{\Wh{A}{a}} \simeq \Z \leftarrow 0 \leftarrow 0 \leftarrow \cdots 
\]
in $\Hotab$. 
\end{coro}
\begin{proof}
Since for any object $a$ in $A$, the category $\tranche{A}{a}$ is aspherical
(\ref{par:defAspherique}), this is a direct corollary of
Proposition~\ref{diag:colimitehomotopiquelibres}.
\end{proof}

\begin{paragr}
If $A$ is a small category, we denote by
\[
  \Z_A: {A}^{\op} \to \Ab \mdvirg a \mapsto \Z
  \]
the constant abelian presheaf of value $\Z$ on $A$, and by~$\H{A}{\Z}$ the
homology of $A$ with coefficients in~$\Z_A$. If $u:A\to B$ is a functor between
small categories, the restriction functor $(u^*)^\ab: \prefab{B} \to
\prefab{A}$ sends $\Z_B$ to $\Z_A$, and induces, by universal property, a
canonical morphism 
\[
\H{u}{\Z}: \H{A}{\Z} \to \H{B}{\Z} 
\]
in $\Hotab$. Commutativity of diagram \ref{diag:colimitehomotopiquelibres}
implies that if $u$ is an element
of $\W_\infty$, then $\H{u}{\Z}$ is an isomorphism in~$\Hotab$. This
means we have defined a canonical homology functor 
\[
\H{-}{\Z}: \Hot \to \Hotab \pbox{.}
\]
\end{paragr}

\begin{remark}
Let $A$ be a small category, and $e_{\pref{A}}$ the final object in the
category $\pref{A}$. Since $i_A(e_{\pref{A}}) \simeq A$, commutativity of
diagram~\ref{diag:colimitehomotopiquelibres} implies that for any small
category $A$, there is a natural isomorphism 
\[
\H{A}{\Z} \simeq \dk\Wh{\Delta}{\nerf{A}}
\]
in $\Hotab$. In other words, we redefined the classical homology
functor for small categories in a way that doesn't give a preeminent
role to $\Delta$. We will give explicit ways to compute
this functor without the use of simplicial methods in Section
\ref{sec:integrators}.
\end{remark}

\begin{remark}\label{remark:homologieTrancheetLibres}
Commutativity of diagram \ref{diag:colimitehomotopiquelibres} can
be reformulated in the following way: for any presheaf~$X$ on $A$,
there is a canonical isomorphism 
\[
\H{A}{\Wh{A}{X}} \simeq \H{\tranche{A}{X}}{\Z} 
\]
in $\Hotab$.
\end{remark}

\begin{paragr}
Let $A$ be a small category. We say that a morphism $f: X \to Y$
between abelian presheaves on~$A$ is a \emph{weak equivalence} if its
image by the functor $\Hf{A}: \prefab{A}\to\Hotab$ is an isomorphism.
We denote by~$\Wab_A$ the class of weak equivalences of abelian
presheaves on $A$, and by 
\[
\Hotab_A:= (\Wab)^{-1}\prefab{A} 
\]
the localization of $\prefab{A}$ by the class of weak equivalences of
abelian presheaves. The functor $\Hf{A}$ then induces a functor 
\[
\overbar{\Hf{A} }: \Hotab_A \to \Hotab
\]
and commutativity of diagram
\ref{diag:colimitehomotopiquelibres} implies that we get a diagram
\[
\xymatrix{
\Hot_A \ar[r]^-{\Whf{A}} \ar[d]_{\overbar{i_A}} 
&
\Hotab_A \ar[d]^{\overbar{\Hf{A}}} 
\\
\Hot \ar[r]_-{\H{-}{\Z}} & \Hotab 
} 
\]
that is commutative up to a natural isomorphism.
\end{paragr}

\begin{defin}\label{def:homPseudoTest}
A small category $A$ is a \emph{homologically pseudo-test category} if
the functor $\overbar{\Hf{A}}$ is an equivalence of categories.
\end{defin}

\begin{remark}
Inspired by the terminology of test categories, the prefix <<pseudo>>
underlines the fact that we do not ask for any conditions on the
quasi-inverse of the functor $\overline{\Hf{A}}$. We introduce
in~\cite{hubert2025phd} a notion of homologically test categories and
homologically test functors that we won't mention in this paper.
\end{remark}

\begin{example}
The prime example of a homologically pseudo-test category is given by the
simplex category~$\Delta$. Recall that the Dold-Kan correspondence
(\cite[Theorem 8.1]{kan1958functors}, \cite[Theorem 1.9]{dold1958homology})
states that the normalized chain complex functor $\dk: \prefab{\Delta} \to
\Ch(\Ab)$ is an equivalence of categories. Moreover, as explained in Example
\ref{ex:normalizedComplexHomology}, this functor computes the homology functor
$\Hf{\Delta}$, which implies that the induced functor between the localized
categories is also an equivalence of categories.
\end{example}

\begin{example}
We will prove in this paper that the categories $\Theta_n$ for every integer
$n>0$, and $\Theta$ are homologically pseudo-test categories.
Using techniques that we do not need to introduce in this paper, we provide in
\cite[Chapter 5]{hubert2025phd} numerous other examples, including the
subcategory $\Delta'$ of monomorphisms of $\Delta$ and the cubical category
$\cub$ without connections.
We also show in \cite{hubert2025phd}, using materials exposed here, that the
category $\Delta^n$ for every integer $n>0$, the cubical category with
connections and the reflexive globular category $\Grefl$ are homologically
pseudo-test categories.
\end{example}

\section{Integrators}\label{sec:integrators}

Grothendieck describes in \cite{pursuingstacks} explicit methods to compute the
homology functor without the need of simplicial techniques and the Bousfield-Kan
formula \cite{bousfield1972homotopy}, that we expose
in this section. While we wont do an extensive study of integrators in this
article (this has been done in \cite{hubert2025phd}), we cannot solely rely on
the Bousfield-Kan formula to study $\Wab_\infty$-asphericity in Section
\ref{sec:asphericalMorphisms}.

\begin{paragr}
If $A$ is a small category, its \emph{additive envelope}, denoted $\Add(A)$, is
the small category defined by the following universal property: for every
additive category $\M$ (that is, an $\Ab$-enriched category with finite
biproducts and a zero object), there is a natural isomorphism 
\[
\Homi(A, \M) \simeq \Homadd(\Add(A),\M) 
\]
where $\Homadd$ denotes the subcategory of functors commuting to finite sums.
For any functor $F: A \to \M$, we denote by $\Add(F)$ the functor making the
following diagram commutative: 
\[
\xymatrix{
A \ar[r]^{F} \ar[d]^{} & \M \\
\Add(A) \ar[ru]_{\Add(F)} &\pbox{.}
} 
\]
\end{paragr}
\begin{paragr}\label{paragr:descrAddEnv}
For any small category $A$, the additive envelope $\Add(A)$ can be realized as
the full subcategory of abelian presheaves on $A$ of the form 
\[
\bigoplus_{i\in I}\Wh{A}{a_i}: a \mapsto \bigoplus_{i\in I}
\Z^{(\Hom_A(a, a_i))}
\]
where $I$ is a \emph{finite} set and $a_i$ is an object of $A$ for every $i$ in
$I$. %
Given any additive category $\M$ and any functor $F
: A \to \M$, the value of $\Add(F)$ on an object $\bigoplus_{i\in
I}\Wh{A}{a_i}$ is then given by the formula
\[
\Add(F)\left(\bigoplus_{i\in I}\Wh{A}{a_i}\right) = \bigoplus_{i\in
I}F(a_i) \pbox{.}
\]
\end{paragr} 

\begin{prop}\label{prop:densityAdd}
For any small category $A$, $\Add(A)$ is a dense subcategory of $\prefab{A}$.
\end{prop}
\begin{proof}
We need to show that for any abelian
presheaf $X$ on $A$, the canonical morphism 
\[
\underset{L \in
\Add(A)/X}{\limind\nolimits^{\prefab{A}}}\!\!\!\!L \to X
\]
is an isomorphism. 
The proof is similar to the proof of the density of
representable presheaves in the $\Set$-valued case, with an additional use of
the diagonal morphisms $(\Wh{A}{a}, X) \to (\Wh{A}{a}\oplus\Wh{A}{a},X)$ of
$\tranche{\Add(A)}{X}$ to check that cocones with vertex $Y$ correspond to
morphisms of \emph{abelian} presheaves from $X$ to $Y$.
We leave the details to the reader, which can also be found in
\cite[Proposition~2.2.4]{hubert2025phd}. \end{proof}

\paragr Given a small category $A$, a cocomplete additive category $\M$ and a
functor~$F: A \to \M$, we get, using Proposition \ref{prop:densityAdd}, a
functor $F_!^\ab: \prefab{A}\to\M$ by setting, for any abelian presheaf $X$ on
$A$,
\[
F_!^\ab(X) = \limind_{L \in \tranche{\Add(A)}{X}} \Add(F)(L) \pbox{.}
\]

\begin{prop}\label{prop:leftKanExtAb}
Let $\M$ be a cocomplete additive category and $A$ a small category. For any
functor~$F: A \to \M$, the functor 
\[
F_!^\ab: \prefab{A}\to\M 
\]
is left adjoint to the functor 
\[
(F^*)^\ab: \M\to\prefab{A} 
\]
defined for every object $x$ of $M$ by 
\begin{align*}
(F^*)^\ab(x): {A}^{\op} &\to \M \\
a &\mapsto \Hom_M(F(a),x) \pbox{.}
\end{align*}
\end{prop}
\begin{proof}
For any object $x$ in $\M$ and any abelian presheaf $X$ on $A$, we
have the following chain of natural isomorphisms:
\begin{align*}
\Hom_M\left(F_!^\ab\,X, x \right)
  & \simeq \Hom_M\big(\!\!\!\!\!
      \sideset{}{^M}\limind_{L \in \tranche{\Add(A)}{X}} \!\!\!\Add(F)(L),\, x
      \big)
  \\ & \simeq \!\!\sideset{}{^\Ab}\limproj_{\substack{L \in \tranche{\Add(A)}{X} \\
    \mathclap{L = \oplus_i\Wh{A}{a_i}}}} 
\Hom_M\Big(\bigoplus_i F(a_i),\, x\Big)
  \\ & \simeq \!\!\sideset{}{^\Ab}\limproj_{\substack{L \in \Add(A)/X\\ 
    \mathclap{L = \oplus_i\Wh{A}{a_i}}}} \prod_i (F^*)^\ab(x)(a_i)
  \\ & \simeq \!\!\sideset{}{^\Ab}\limproj_{\substack{L \in \tranche{\Add(A)}{X} \\ 
    \mathclap{L = \oplus_{i}\Wh{A}{a_i}}}} \prod_i
\Hom_{\prefab{A}}\big(\Wh{A}{a_i}, (F^*)^\ab(x)\big)
  \\ & \simeq \Hom_{\prefab{A}}\Big(\sideset{}{^{\prefab{A}}}\limind_{L \in
      \Add(A)/X}\!\!\!L,\,\,
      (F^*)^\ab(x)\Big)
  \\ & \simeq \Hom_{\prefab{A}} \left(X, (F^*)^\ab(x)\right) \pbox{.}
  \qedhere
\end{align*}
\end{proof}

\begin{coro}\label{coro:F_!cocontinu}
Let $\M$ be a cocomplete additive category and $A$ a small category.
For any functor $F: A \to
\M$, the functor $F_!^\ab: \prefab{A}\to\M$ is cocontinuous.
\end{coro}

\begin{prop}\label{prop:prefabEqCat}
The restriction functor 
\[
\Homi_!(\prefab{A},\M) \to \Homi(A,\M) \mdvirg F \mapsto F \circ
\Whf{A}
\]
is an equivalence of categories, and a quasi-inverse is given by the
functor
\[
\Homi(A,\M)\to\Homi_!(\prefab{A},\M) \mdvirg F \mapsto F_!^\ab \pbox{,}
\]
where $\Homi_!$ denotes the subcategory of cocontinuous functors.
\end{prop}
\begin{proof}
By construction, if $a$ is an object of $A$, we have a natural
isomorphism $F_!^\ab(\Z^{(a)}) \simeq F(a)$. Moreover, if $F:
\prefab{A}\to \M$ is a cocontinuous functor, we can check that $F$ and
$(F\circ\Whf{A})_!^\ab$ agree by checking on representables, which is
straightforward.
\end{proof}

\paragr Let $A$ be a small category. Denote by 
\[
- \odot_A -: \prefab{A} \times \prefab{{A}^{\op}} \to \Ab 
\]
the bifunctor defined by the commutative diagram
\[
\xymatrix{
\prefab{A}\times\prefab{{A}^{\op}} 
\ar[r]^-{-\odot_A -} 
\ar[d]_{\id \times (-)_!^\ab}
&
\Ab \\
\prefab{A} \times \Homi_!(\prefab{A},\Ab)
\ar[ru]_{\ev}
& \pbox{,}
}
\]
where we used the fact that $\prefab{{A}^{\op}}=\Homi({A},\Ab)$. 
In other words, for every abelian presheaf $X$ on $A$ and every
abelian presheaf $Y$ on ${A}^{\op}$, we have the following identity: 
\[
X \odot_A Y = Y_!^\ab(X) \pbox{.} 
\]

\begin{prop}\label{prop:tensorProductRepresentableLeft}
Let $a$ be an object of a small category $A$ and $Y$ an abelian
presheaf on ${A}^{\op}$. There is a canonical isomorphism 
\[
\Wh{A}{a} \odot_A Y \simeq Y(a) 
\]
in $Ab$.
\end{prop}
\begin{proof}
This is already contained in Proposition \ref{prop:prefabEqCat}.
\end{proof}

\begin{remark}
The latter proposition, together with cocontinuity, is enough to see
that the bifunctor~$- \odot_A -$ coincide with the tensor product of
functors: for every abelian presheaf $X$ on $A$ and every presheaf $Y$ on
${A}^{\op}$, we have a natural isomorphism of abelian groups
\[
X \odot_A Y \simeq \int^{a \in A} Xa \otimes Ya \pbox{.}
\]
While we could have used this description in this article, we chose to expose the formalism
introduced by Grothendieck in \cite{pursuingstacks}. 
We will still, however, refer to this bifunctor as the tensor product of
functors.
\end{remark}

\begin{prop}\label{prop:tensorProductSymmetry}
Let $X$ be an abelian presheaf on $A$ and $Y$ an abelian presheaf on
${A}^{\op}$. There is a natural isomorphism of abelian groups 
\[
X \odot_A Y \simeq Y \odot_{{A}^{\op}} X \pbox{.}
\]

\end{prop}
\begin{proof}
The functor $- \odot_A -$
is cocontinuous on the first variable by Corollary
\ref{coro:F_!cocontinu}. Furthermore, for any abelian presheaf $X$ on
$A$, the functor
\[
  X \odot_A -: \prefab{{A}^{\op}} \to \Ab
  \]
is the composite of the equivalence of categories given by Proposition
\ref{prop:prefabEqCat} and the evaluation functor, both of which are also
cocontinuous.
The situation is identical for the functor~$- \odot_{{A}^{\op}} -$, therefore we
only have to check the isomorphism with representables on both arguments. For
any object $a$ in $A$ and any object $a'$ in
${A}^{\op}$, we have 
\[ \Wh{A}{a} \odot_A \Wh{{A}^{\op}}{a'} \simeq
\Z^{(\Hom_{{A}^{\op}}(a,a'))}
\]
and
\[  \Wh{A^{\op}}{a'} \odot_{A^{\op}} \Wh{A}{a} \simeq
\Z^{(\Hom_A(a',a))}  
\]
which concludes the proof.
\end{proof}

In the rest of this section, we will show how the tensor product of
functors relates to the homology functor we introduced in Section
\ref{sec:homology}.

\begin{paragr}
Recall that for any category $A$, the category $\prefab{A}$ is an
abelian category with enough projective objects. To describe
projective objects, we use the \emph{infinite additive envelope},
denoted $\Addinf(A)$, defined by the following universal property:
for every additive category $\M$, there is a natural isomorphism 
\[
\Homi(A,\M) \simeq \Homaddinf(\Addinf(A),\M) 
\]
where $\Homaddinf$ denotes the subcategory of functors commuting to arbitrary
sums. 
It can be realized as the full subcategory of $\prefab{A}$ whose objects
are arbitrary sums of representables.
\end{paragr}

\begin{prop}\label{prop:resAddinf}
Let $A$ be a small category. \begin{enumerate}
\item Every object of $\Addinf(A)$ is a projective object in
$\prefab{A}$.
\item Every object in $\prefab{A}$ admits a projective resolution by
objects of $\Addinf(A)$.
\item Projectives objects in $\prefab{A}$ are retracts of objects of
$\Addinf{A}$.
\end{enumerate} 
\end{prop}
\begin{proof}
The proof of $(1)$ is straightforward since coproducts of projectives
are projectives. For $(2)$, given an abelian presheaf $X$ on $A$, we
set 
\[
  L_X=\bigoplus_{a\in\Ob(A)}\bigoplus_{x\in Xa}\Wh{A}{a}
\]
and we define an epimorphism $\varphi: L_X \to X$ on each summand
$(a,x)$ by defining $\varphi_{(a,x)}$ to be the element of $Xa$ represented
by the morphism $\Wh{A}{a}\xrightarrow{x} X$. The proof of the
last point then follows from the fact that if $X$ is a projective
object in $\prefab{A}$, the morphism~$\varphi: L_X \to X$ admits a
section, exhibiting $X$ as a retract of $L_X$.
\end{proof}

\begin{prop}\label{produittensorielderive}
Let $X$ be an abelian presheaf on $A$ and $Y$ an abelian presheaf on
${A}^{\op}$. The functors
\begin{align*}
X \odot_A -: \prefab{{A}^{\op}} &\to \Ab  \\
- \odot_A Y: \prefab{A} &\to \Ab
\end{align*}
both admit left derived functors
\begin{align*}
L(X \odot_A -): \mathcal{D}_+(\prefab{{A}^{\op}}) &\to \Hotab \mdvirg \\
L(- \odot_A Y): \mathcal{D}_+(\prefab{A}) &\to \Hotab \mdvirg
\end{align*}
and there is a natural isomorphism
\[
L(X \odot_A -)(Y) \simeq L(- \odot_A Y)(X) 
\]
in $\Hotab$.
\end{prop}
\begin{proof}
Consider two projective resolutions~$M_{\bullet}\xrightarrow{\epsilon} X$ and
$L_{\bullet} \xrightarrow{\eta} Y$ by objects of $\Addinf(A)$ and
$\Addinf({A}^{\op})$ respectively. 
We can then form the chain complexes of abelian groups $M_{\bullet}\odot_A Y$,
$X\odot_A L_{\bullet}$, as well as the double complex~$M_{\bullet}\odot_A
L_{\bullet}$ obtained by applying the bifunctor $\odot_A$ term by term.
Considering the complexes $M_{\bullet}\odot_A Y$ and $X\odot_A L_{\bullet}$
as double complexes concentrated in the first column or in the
first row, the resolutions $\epsilon$ and $\eta$ provide morphisms of
double complexes which in turn induce morphisms of complexes 
\[
\Tot (M_\bullet \odot_A Y) \xleftarrow{M \odot_A \eta} \Tot(M_\bullet
\odot_A L_\bullet)
\xrightarrow{\epsilon\odot_A L} \Tot(X \odot_A L_\bullet) \pbox{.}
\]
We will use a classical argument from homological algebra (see for example
\cite[2.7]{weibel1994Introduction}) to show that these morphisms are
quasi-isomorphisms. For this, we only need to show that the functors 
\[
- \odot_A L_j \mdvirg M_i \odot_A - \mdvirg i,j\in \N
\]
are exact. By symmetry (Proposition \ref{prop:tensorProductSymmetry}), it is
enough to show that $N \odot_A -$ is exact for every
object $N=\bigoplus_{i\in I}\Wh{A}{a_i}$ in $\Addinf(A)$. Since for every
abelian presheaf $Y$ on $B$ we have
\[
\left(\bigoplus_{i\in I}\Wh{A}{a_i}\right)\odot_A Y = \bigoplus_{i\in I}Y(a_i) 
\pbox{,}
\]
and since the evaluation and direct sum functors are exact, we can conclude.
\end{proof}

\begin{lemme}
Let $A$ be a small category. If $X$ is an abelian presheaf
on $A$, there is a natural isomorphism of abelian groups
\[
X \odot_A \Z_{{A}^{\op}} \simeq \limind X \pbox{.}
\]
\end{lemme}
\begin{proof}
Given an abelian presheaf $X$ on $A$, we have the following chain of
natural isomorphisms in $\Ab$:
\begin{align*}
X \odot_A \Z_{A^{\op}} 
& \simeq X \odot_A \Wh{{A^{\op}}}{ e_{\pref{{A^{\op}}}} }\\
& \simeq X \odot_A 
\Wh{{A^{\op}}}{ \limind\nolimits_{a\in
{A^{\op}}}^{\pref{{A^{\op}}}} a }
\\ & \simeq X \odot_A \limind\nolimits_{a \in {A^{\op}}}^{\prefab{{A^{\op}}}}
\Wh{{A^{\op}}}{a} 
\\ & \simeq \limind\nolimits_{a \in {A^{\op}}}^{\Ab} 
\left( X \odot_A \Wh{{A^{\op}}}{a} \right)
\\ & \simeq \limind\nolimits_{a \in {A^{\op}}}^{\Ab} X(a) \pbox{.}
\qedhere
\end{align*}
\end{proof}

\begin{coro}
Let $A$ be a small category, and let $L_\bullet$ be a projective resolution of
$\Z_{{A}^{\op}}$ in the category $\prefab{{A}^{\op}}$. For any abelian presheaf
$X$ on $A$, there is a canonical natural isomorphism 
\[
\H{A}{X} \simeq X \odot_A L_\bullet
\]
in $\Hotab$.
\end{coro}
\begin{proof}
Since we defined the functor $\Hf{A}:\prefab{A}\to\Hotab$ to be the left derived
functor of the colimit functor (\ref{paragr:HomologyFunctor}), there is nothing
else to prove.
\end{proof}

\begin{defin}
An \emph{integrator} on $A$ is a projective resolution 
\[
L = \left( L_0 \xleftarrow{d_0} L_1 \xleftarrow{d_1} L_2 \xleftarrow{d_2} \cdots
\right) \mdvirg L
\xrightarrow{\epsilon} \Z_{{A}^{\op}}
\]
of the constant abelian presheaf $\Z_{{A}^{\op}}$ of value $\Z$ on ${A}^{\op}$.
This means that each $L_n$ is a projective presheaf on ${A}^{\op}$, and that
$\epsilon_0 : L_0 \to \Z$ is the cokernel of $d_0$.
\end{defin}

\paragr We say that $L$ is a \emph{free} integrator if it is an integrator whose
components $L_n$ are in $\Addinf({A}^{\op})$ for every integer $n \geq 0$. This
allows for very simple computations thanks to the following proposition.

\begin{prop}\label{prop:tensorProductAddRight}
If $X$ is an abelian presheaf on $A$ and if $L = \bigoplus_{i\in I}
\Wh{{A}^{\op}}{a_i}$ is an object of $\Addinf({A}^{\op})$, there is a natural
isomorphism 
\[
X \odot_A L \simeq \bigoplus_{i\in I} X(a_i)
\]
in $\Ab$.
\end{prop}
\begin{proof}
We can use the symmetry proved in proposition \ref{prop:tensorProductSymmetry}
together with proposition \ref{prop:tensorProductRepresentableLeft}.
\end{proof}

\begin{remark}
Since $\Ch(\prefab{{A}^{\op}}) = \Homi(A,\Ch(\Ab))$, an integrator $L$ on a
small category $A$ can be seen either as a complex in $\Add({A}^{\op})$, or as a
functor $A \to \Ch(\Ab)$. We will use interchangeably the two points of view.
\end{remark}

\begin{example}\label{ex:IntegratorDelta}
Taking for $A$ the category $\Delta$, let
\[
L_\Delta = 
\left(
\Wh{\Delta^{\op}}{\Delta_0} 
\leftarrow \Wh{\Delta^{\op}}{\Delta_1} 
\leftarrow \Wh{\Delta^{\op}}{\Delta_2} 
\leftarrow \cdots
\right)
\]
be the complex in $\Add({\Delta}^{\op})$ whose differential is given by the
alternate sum of cofaces. 
For every integer $n\geq 0$, $\Wh{\Delta^{\op}}{\Delta_n}$ is the
abelian presheaf on ${\Delta}^{\op}$
\[
  \Wh{{\Delta}^{\op}}{\Delta_n}: \Delta_k
  \mapsto \\Z^{(\Hom_\Delta(\Delta_n,\Delta_k))} \pbox{,}
  \]
hence the direction of the arrows. 
Given any abelian simplicial group $X$, the complex $X \odot (L_\Delta)_\bullet$
is, by Proposition \ref{prop:tensorProductAddRight}, the non-normalized chain complex associated to $X$
\[
X \odot (L_\Delta)_\bullet \simeq
\left(
X_0 \leftarrow X_1 \leftarrow X_2 \leftarrow \cdots 
\right)
\]
with differential given by the alternative sum of faces.
We claim that the complex $L_\Delta$ is a free integrator on $\Delta$.
Proving this amounts to compute the homology of representables, which
is a classical exercise. Details can be found for example in
\cite[Example~2.5.7]{hubert2025phd}. Thus, we get another proof of the
fact that the non-normalized chain complex computes homology, as defined
in paragraph \ref{paragr:HomologyFunctor}. Since the normalized and
non-normalized chain complex functors are pointwise quasi-isomorphic, this also
gives another proof of the fact that $\dk : \prefab{\Delta}\to\Ch(\Ab)$ computes
the homology functor.
\end{example}

\begin{paragr}\label{paragr:BousfieldKanIntegrator}
Any small category $A$ can be equipped with a free integrator that we
call the \emph{Bousfield-Kan integrator}, defined as the composite
\[
\ell_A: A \hookrightarrow \pref{A} \xrightarrow{i_A} \Cat
\xrightarrow{\nerf} \pref{\Delta} \xrightarrow{\Whf{\Delta}} 
\prefab{\Delta} \xrightarrow{{L_{\Delta}}_!^\ab} \Ch(\Ab) \pbox{.}
\]
Indeed, for any integer $n\geq 0$, we have
  \[
{(\ell_A)}_n = \bigoplus_{\Delta_n
\xrightarrow{u}A}\Wh{{A}^{\op}}{\varphi n}: a \mapsto
\bigoplus_{\Delta_n \xrightarrow{u}A} \Z^{(\Hom_A(\varphi n,a))}
  \]
which shows that $\ell_A$ is a complex with components in
$\Add({A}^{\op})$. 
Moreover, given any object $a$ of~$A$, the category $\tranche{A}{a}$ is
aspherical, which means that the functor $\tranche{A}{a}\to e$ is in
$\W_\infty$. Applying the nerve and the functor ${L_\Delta}_!^\ab$ then gives a
natural quasi-isomorphism $\epsilon : \ell_A(a) \to
{L_\Delta}_!^\ab\Wh{{\Delta}^{\op}}{\nerf(e)}$. The codomain of this morphism is
the complex of abelian groups 
\[
{L_\Delta}_!^\ab\Wh{{\Delta}^{\op}}{\nerf(e)} =
\left(
\Z \xleftarrow{0} \Z \xleftarrow{\id} \Z \xleftarrow{0} \Z \xleftarrow{\id}
\cdots
\right) \pbox{,}
\]
which we equip with the quasi-isomorphism to $\Z[0]$ given by the
identity in degree $0$. Composing these morphisms, we get a natural
quasi-isomorphism $\ell_A(a) \to \Z[0]$ which is given by the augmentation
morphism $\ell_A(a)_0 \to \Z$.
This proves that $\ell_A$ is a free integrator on $A$.

Given any abelian presheaf $X$ on $A$, the complex $(\ell_A)^\ab_!(X)$ is
then the non-normalized complex of the simplicial replacement
\cite{bousfield1972homotopy} of $X$:
\[
(\ell_A)^\ab_!(X) \simeq 
\Big(
\cdots \leftarrow
\bigoplus_{\Delta_n \xrightarrow{u} A} Xu(n) 
\leftarrow
\bigoplus_{\Delta_{n+1} \xrightarrow{u} A} Xu(n+1) 
\leftarrow \cdots
\Big)
\]
whose differential is given on generators by 
\[
  d \langle \Delta_{n+1} \xrightarrow{u} A, x\rangle = 
  \sum_{0 \leq i \leq n+1} (-1)^i \langle \Delta_{n}
  \xrightarrow{\delta_i} \Delta_{n+1} \xrightarrow{u} A, d_i(x) \rangle
\]
where we denoted by $\delta_i$ the $i$-th coface maps and by $d_i$ its image
by $X$. Using this integrator amounts to use the classical
Bousfield-Kan formula given in \cite[XII, 5.5]{bousfield1972homotopy}.
\end{paragr}

A detailed study of integrators is not needed for the purpose of this
paper, since we will mainly (but not exclusively) use the Bousfield-Kan
integrator. More details are available in \cite{hubert2025phd} where we give,
for example, integrators on the cubical category $\cub$ and the globular category
$\G$. We also use this notion to prove several homotopical Dold-Kan
correspondences for categories that do not satisfy the strong Whitehead
condition mentioned in the introduction.

\section{Aspherical morphisms for homology}
\label{sec:asphericalMorphisms}

It is fairly easy to see from Proposition
\ref{prop:MorphismesAspheriquesLocFond} that if $u : A \to B$ is an aspherical
functor, the restriction functor $u^* : \pref{B} \to \pref{A}$ induces an
isomorphism in homology for presheaves of sets. In this section, we show that it
is enough to ask for $u$ to be aspherical with respect to the class of homology
isomorphisms in $\Cat$, and that the restriction functor also induces an
isomorphism in homology for arbitrary abelian presheaves.

\begin{paragr}
We denote by $\Wab_\infty$ the class of functors in $\Cat$ inducing an
isomorphism in homology. More precisely, a functor $u: A \to B$
between small categories is an element of $\Wab_\infty$ if and only if
the induced morphism 
\[
\H{A}{\Z} \xrightarrow{\H{u}{\Z}} \H{B}{\Z} 
\]
is an isomorphism in $\Hotab$. Using the standard way of computing
homology of categories with the simplicial nerve functor, we get the following
equivalence: 
\[
u \in \Wab_\infty \iff \dk\Wh{\Delta}{\nerf(u)} \text{ is a quasi-isomorphism.}
\]
\end{paragr}

\begin{prop}\label{prop:asphericalAbAspherical}
The class $\Wab_\infty$ is a basic localizer. Moreover, we have the
inclusion $\W_\infty \subset \Wab_{\infty}$, and every aspherical
functor is $\W_\infty^\ab$-aspherical.
\end{prop}

\begin{proof}
This follows (see \cite[Corollary
1.5.14]{maltsiniotis2001derivateurs}) from the fact that $\Wab_\infty$
is the class of equivalences for the derivator induced by the
localizer $(\Ch(\Ab), \W_{qis})$ where $\W_{qis}$ is the class of
quasi-isomorphisms of non-negative chain complexes of abelian groups.
More details (in French) can be found in
\cite[annexe~A]{hubert2025phd}. The inclusion $\W_\infty \subset
\Wab_\infty$ can be seen directly from the equivalence above, or,
without using the simplicial description, by the minimality of
$\W_\infty$ among basic localizers \cite{cisinski2004localisateur}.
\end{proof}

\begin{paragr}
Let $u : A \to B$ be a functor between small categories. Recall that the functor
$u^* : \pref{B} \to \pref{A}$ is a right adjoint. In particular, it preserves
abelian group objects and induces a functor that we denote
\[
(u^*)^\ab : \prefab{B} \to \prefab{A} \pbox{.}
\]
The functor $u^*$ also admits a right adjoint $u_* : \pref{A} \to \pref{B}$,
which, by the same argument, induces a functor
\[
u_*^\ab : \prefab{A} \to \prefab{B}
\]
which is right adjoint to the functor $(u^*)^\ab$.
The two following squares are then commutative up to a natural isomorphism:
\begin{equation}
\label{diag:invImageCommutationsOubli}
\xymatrix{
\prefab{B} \ar[d]_{\U} \ar[r]^{(u^*)^\ab} &
\prefab{A} \ar[d]^{\U} \\
\pref{B} \ar[r]_{u^*} & \pref{A} \pbox{,}
} 
\qquad
\xymatrix{
\prefab{A} 
\ar[r]^{u_*^\ab}  
\ar[d]_{\U}  
& 
\prefab{B} 
\ar[d]^{\U} 
\\
\pref{A}  
\ar[r]_{u_*}
& \pref{B} \pbox{.}
}
\end{equation}

The left adjoint to $u^*$, however, does not preserve abelian group objects
in general. But since $\prefab{B}$ is a cocomplete additive category, Proposition
\ref{prop:leftKanExtAb} allows to extend the functor $A \xrightarrow{u} B
\xrightarrow{\Whf{B}} \prefab{B}$ to a cocontinuous functor
\[
  (\Whf{B}\circ u)_!^\ab : \prefab{A} \to \prefab{B}
\]
that we will simply denote by $u_!^\ab$. Proposition \ref{prop:leftKanExtAb}
guarantees that this functor is a left adjoint to the functor
\[
( (\Whf{B}\circ u)^*)^\ab : \prefab{B} \to \prefab{A} \mdvirg 
Y \mapsto \left( a \mapsto \Hom_{\prefab{B}}(\Wh{B}{u(a)},Y) \right)
\]
which is nothing but the functor $(u^*)^\ab$. 
Commutativity of the squares in diagram \ref{diag:invImageCommutationsOubli}
imply that the following two squares 
\begin{equation}
\label{diag:invImageCommutationWhitehead}
\xymatrix{
\pref{B} \ar[d]_{\Whf{B}} \ar[r]^{u^*} & \pref{A} \ar[d]^{\Whf{A}} 
\\
\prefab{B} \ar[r]_{(u^*)^\ab} & \prefab{A} \pbox{,}
}
\qquad
\xymatrix{
\pref{A} \ar[d]_{\Whf{A}} \ar[r]^{u_!} & \pref{B} \ar[d]^{\Whf{B}} 
\\
\prefab{A} \ar[r]_{u_!^\ab} & \prefab{B} 
}
\end{equation}
are also commutative up to a natural isomorphism.
\end{paragr}

\begin{paragr}
Let $A$ and $B$ be two small categories, and $u : A \to B$ a
$\Wab_\infty$-aspherical functor. Recall from Proposition
\ref{prop:MorphismesAspheriquesLocFond} that the functor 
\[
\tranche{u}{X}:\tranche{A}{X}\to \tranche{B}{X} \mdvirg (a, u(a)\to X) \mapsto
(u(a),u(a)\to X)
\]
is then in $\Wab_\infty$. Thus, taking homology with integer coefficients, the morphism
\[
  \H{\tranche{u}{X}}{\Z} : \H{\tranche{A}{X}}{\Z} \to \H{\tranche{B}{X}}{\Z}
\]
is an isomorphism in $\Hotab$. Using Remark
\ref{remark:homologieTrancheetLibres} together with the commutativity of diagram
\ref{diag:invImageCommutationWhitehead}, this means that there is a natural
isomorphism 
\[
\H{A}{(u^*)^\ab\Wh{B}{X}} \simeq \H{B}{\Wh{B}{X}}
\]
in $\Hotab$. In other words, $(u^*)^\ab$ induces an isomorphism in homology for
\emph{free} abelian presheaves. We will now show that this still holds for
arbitrary abelian presheaves.
\end{paragr}

\begin{paragr}
For any object $b$ of $B$, we denote by $\ell_A^B: B \to \Ch(\Ab)$
the functor making the following diagram 
\[
\xymatrix{
B \ar[r]^{\Whf{B}} \ar[rd]_{\ell_A^B} 
& \prefab{B} \ar[r]^{(u^*)^\ab} 
& \prefab{A} \ar[ld]^{(\ell_A)_!^\ab} 
\\
& \Ch(\Ab)
} 
\]
commutative, where $\ell_A$ denotes the Bousfield-Kan integrator introduced in
paragraph \ref{paragr:BousfieldKanIntegrator}. Concretely, for any object $b$ of
$B$, $\ell_A^B(b)$ is the complex which is given in degree~$n\geq0$ by 
\[
(u^*)^\ab\Wh{B}{b} \odot_A (\ell_A)_n = 
\bigoplus_{\Delta_n \xrightarrow{\varphi}A}\Wh{}{\Hom_B(u\varphi(n),b)} \pbox{.}
\]
In particular, $\ell_A^B$ is a complex in $\Add({B}^{\op})$. Moreover, we get
the equality
\[
  \ell_A^B(b) = (L_\Delta)_!^\ab\Wh{\Delta}{\nerf(\tranche{A}{b})} \pbox{.}
\]
Therefore, applying the functor
\[
\Cat \xrightarrow{\nerf} \pref{\Delta} \xrightarrow{\Whf{\Delta}}
\prefab{\Delta}
\xrightarrow{(L_\Delta)_!^\ab} \Ch(\Ab) 
\]
to $\tranche{u}{b}: \tranche{A}{b}\to\tranche{B}{b}$ gives a morphism of chain complexes $\lambda_{u,b}:
\ell_A^B(b)
\to \ell_B(b)$ which is natural in~$b$, since for any morphism $f: b \to b'$ in
$B$, the square 
\[
\xymatrix{
\tranche{A}{b} \ar[r]^{\tranche{u}{b}} 
\ar[d]_{\tranche{A}{f}} 
&
\tranche{B}{b} \ar[d]^{\tranche{B}{f}} 
\\
\tranche{A}{b'} \ar[r]_{\tranche{u}{b'}} 
&
\tranche{B}{b'}
} 
\]
is commutative. Given any abelian presheaf $X$ on $B$, we can then
define a morphism
\[
  \lambda_{u,X} = X \odot_B \lambda_u: X \odot_B \ell_A^B \to X
  \odot_B \ell_B \mdvirg
\]
which fits in the following diagram:
\begin{equation}\label{diag:morphismeWAspheriqueCommutationLambda}
\xymatrix{
{\prefab{B}} 
\ar[rr]^{(u^*)^\ab} 
\ar[rd]_{(\ell_B)_!^\ab}^{}="c"
&& 
{\prefab{A}} 
\ar@{}"c"|(.4){}="d"|(.9){}="e"
\ar@2"d";"e"_{\lambda_u}
\ar[ld]^{(\ell_A)_!^\ab} \\
& {\Ch(\Ab)} 
} 
\end{equation}
and which is described in the following way: for any abelian presheaf
$X$ on $B$ and any integer~$n\geq 0$, we have 
\[
X \odot_B (\ell_A^B)_n = 
\bigoplus_{\Delta_n \xrightarrow{\varphi}A}X(u\varphi(n))
\mdvirg
X \odot_B (\ell_B)_n = 
\bigoplus_{\Delta_n \xrightarrow{\varphi}B}X(\varphi(n)) \mdvirg
\]
and $(\lambda_{u,X})_n$ is given on generators by the formula
$(\lambda_{u,X})_n \langle \varphi, x \rangle = 
\langle u \circ \varphi , x \rangle$.
\end{paragr}

\begin{prop}
Let $u: A \to B$ be a functor between small categories. The following
conditions are equivalent: 
\begin{enumerate}
\item $u$ is $\Wab_\infty$-aspherical;
\item for any object $b$ in $B$, the morphism $\lambda_{u,b}$ is a
quasi-isomorphism;
\item for any abelian presheaf $X$ on $B$, the morphism
$\lambda_{u,X}$ is a quasi-isomorphism;
\item the diagram 
\[
\xymatrix{
\prefab{B} \ar[rr]^{(u^*)^\ab} \ar[rd]_{\Hf{B}} && \prefab{A} \ar[ld]^{\Hf{A}}
\\
& \Hotab
} 
\]
is commutative up to a natural isomorphism.
\end{enumerate}
Moreover, these conditions imply the following condition 
\begin{enumerate}[resume]
\item for any morphism $f$ in $\prefab{B}$, we have the equivalence 
\[
f \in \Wab_B \iff (u^*)^\ab f \in \Wab_A \pbox{.}
\]
\end{enumerate}
\end{prop}
\begin{proof}
The implication $(4) \implies (5)$ is immediate. To prove
$(1)\implies(2)$, recall that if $u$ is $\Wab_\infty$\nobreakdash-aspherical,
Proposition \ref{prop:MorphismesAspheriquesLocFond} implies that for
any object $b$ in $B$, the morphism $\tranche{u}{b}$ is in
$\Wab_\infty$. By construction, $\lambda_{u,b}$ is then a
quasi-isomorphism. 
For the implication $(2)\implies(3)$, recall that $\ell_A^B$ is a
complex in $\Add({B}^{\op})$
which implies, with the hypothesis, that $\lambda_{u}:
\ell_A^B\to\ell_B$ is a quasi-isomorphisms between complexes of projective
abelian presheaves on ${B}^{\op}$. Since for any abelian
presheaf $X$ on $B$, the functor $X \odot_B -$ is a left adjoint, it
is right exact, hence preserves quasi-isomorphisms between
complexes of projective objects. Therefore, $\lambda_{u,X} = X \odot_B
\lambda_u$
is a quasi-isomorphism. Implication $(3) \implies (4)$ comes from the
fact that $\ell_B$ is an integrator on $B$, so that with this
hypothesis, $\ell_A^B$ is also an integrator on $B$. Finally, to prove
$(4)\implies(1)$, we use Proposition
\ref{prop:MorphismesAspheriquesLocFond} and prove that for any object
$b$ of $B$, the category $\tranche{A}{b}$ is $\Wab_\infty$-aspherical:
this comes from the isomorphisms 
\[
\H{\tranche{A}{b}}{\Z} \simeq \H{A}{\Wh{A}{u^*(b)}} \simeq
\H{B}{\Wh{B}{b}} \simeq \H{\tranche{B}{b}}{\Z}
\]
in $\Hotab$, and from the fact that $\tranche{B}{b}$ is
a $\Wab_\infty$-aspherical category. 
\end{proof}

\section{Strong Whitehead categories}
\label{sec:whitehead}

There are two natural candidates to the role of weak equivalences of
abelian presheaves on a small category: morphisms inducing isomorphisms in homology, and morphisms
whose underlying morphism of presheaves of sets is a weak equivalence. We show
in this section that any test category for which these two classes coincide is a
homologically pseudo-test category, as defined in~Definition
\ref{def:homPseudoTest}.

\begin{defin}\label{def:Whitehead}
We say that a small category $A$ is a \emph{strong Whitehead} category if the
two classes of morphisms~$\Wab_A$ and $\U^{-1}\W_A$ in $\prefab{A}$
coincide, where $\U: \prefab{A} \to \pref{A}$ denotes the forgetful
functor. This means that for any morphism $f$ of abelian presheaves on $A$, we have the
equivalence 
\[
f \in \Wab_A \iff \U(f) \in \W_A \pbox{.}
\]

\end{defin} 
\begin{example}\label{ex:DeltaWhitehead}
Recall
that for any morphism of simplicial abelian groups $f: X \to Y$ and
any integer $n>0$, there is a commutative square 
\[
\xymatrix{
{\pi_n(\U X, 0)} \ar[r]^{\pi_n(\U f)} \ar[d]_{\simeq}  & {\pi_n(\U Y, 0)}
\ar[d]^{\simeq}  \\
\mathsf{H}_n(\Delta,\Z^{(X)}) \ar[r]_{\mathsf{H}_n(f)} &
\mathsf{H}_n(\Delta,\Z^{(Y)})
}
\]
and similarly in degree $0$. This shows that $f$ is in $\Wab_A$ if
and only if $\U f$ induces an isomorphism between homotopy groups.
Illusie-Quillen Theorem \cite[VI, Theorem 3.3]{quillenktheory} then
allows to conclude that $\Delta$ is a strong Whitehead category.
\end{example}

\begin{prop}\label{prop:aspheriqueSourceWhitehead}
Let $u: A \to B$ be an aspherical functor. If $A$ is a strong Whitehead
category, then $B$ is also a strong Whitehead category.
\end{prop}
\begin{proof}
If $u$ is aspherical, then it is also $\Wab_\infty$-aspherical by Proposition
\ref{prop:asphericalAbAspherical}, and we get the sequence of equivalences 
\begin{align*}
f \in \Wab_B &\iff (u^*)^\ab f \in \Wab_A
& (\text{$\Wab_\infty$-asphericity}) \\
&\iff \U (u^*)^\ab f \in \W_A 
& (\text{strong Whitehead condition})\\
&\iff u^* \U f \in \W_A 
& (\text{commutativity of diagram \ref{diag:invImageCommutationsOubli}})\\ 
&\iff \U f \in \W_B &\text{(asphericity)}
\end{align*}
which shows that $B$ is a strong Whitehead category.
\end{proof}

\begin{example}
\label{ex:WhiteheadCategories}
Since $\Delta$ is totally aspherical \cite[Proposition
1.6.14]{maltsiniotis2005}, the diagonal functor $\Delta \to \Delta\times\Delta$
is aspherical. This implies that for any integer $n > 0$, the category
$\Delta^n$ is also a strong Whitehead category. We will see in Proposition
\ref{prop:m_nAspherical} that there exists an aspherical functor
$\Delta^n \to \Theta_n$, ensuring that $\Theta_n$ is also a strong Whitehead
category for any integer $n>0$. \end{example}

\begin{paragr}
If $\M$ is a locally presentable category and $I$ is a class of
morphisms in $\M$, we denote by~$l(I)$ (resp.~$r(I)$) the class of
morphisms satisfying the left (resp.~right) lifting property with
respect to~$I$. We recall the following theorem:
\end{paragr}

\begin{theorem}[Crans]\label{th:CransTransfert}
Let $(\M, \W_\M, \textit{Cof}_\M,\textit{Fib}_\M)$ be a closed model
category cofibrantly generated by the pair $(I,J)$. Let $\M'$ be a
locally presentable category and $G: \M \to \M'$ a functor admitting
a right adjoint $D$. Suppose that 
$D(l(r(GJ)))\subset \W_\M$. 
Then $\M'$ admits a closed model category structure generated by the
pair $(GI,GJ)$, whose weak equivalences are the elements of
$D^{-1}(\W_\M)$ and fibrations are the elements of
$D^{-1}(\textit{Fib}_\M)$.
\end{theorem}
\begin{proof}
See \cite[Theorem 3.3]{Crans1995Quillen}.
\end{proof}

In what follows, we will use Theorem \ref{th:CransTransfert} to transfer the
Grothendieck-Cisinski model structure~(\ref{thm:modelStructureCisinski}) on
$\pref{A}$ to the category $\prefab{A}$ in the case where $A$ is both a strong
Whitehead and a local test category.

\begin{prop}\label{prop:intLibreMonos}
Let $A$ be a small category and $i: X \hookrightarrow Y$ a monomorphism of
presheaves on~$A$. If~$L: A \to \Ch(\Ab)$ is a free integrator on
$A$, then the morphism
\[
L_!^\ab ( \Wh{A}{i} ): L_!^\ab(\Wh{A}{X}) \to
L_!^\ab(\Wh{A}{Y})
\]
is a monomorphism with projective cokernel.
\end{prop}
\begin{proof}
Since $L$ is a free integrator, we can write, for any integer $n \geq
0$, 
\[
L_n = \bigoplus_{i\in I_n}\Wh{{A}^{\op}}{a_i} 
\]
and, by additivity, we only have to show that if $a$ is an object of
$A$, then the morphism of abelian groups
\[
\Wh{A}{i} \odot \Wh{{A}^{\op}}{a}: 
\Wh{A}{X} \odot_A \Wh{{A}^{\op}}{a} 
\to
\Wh{A}{Y} \odot_A \Wh{{A}^{\op}}{a}
\]
 is a monomorphism with projective cokernel. But this is the free
 abelian group morphism
\[
  \Z^{(i_a)}: \Z^{(Xa)}\to\Z^{(Ya)}
  \]
which is a monomorphism with cokernel the free abelian group on the set
$Ya\setminus \im i_a$.
\end{proof}

\begin{coro}\label{homologielibrecofibtriviale}
Let $A$ be a local test category, $j: X \to Y$ a trivial cofibration
for the Grothendieck-Cisinski model category structure on $\pref{A}$,
and $L: A \to \Ch(\Ab)$ a free integrator on $A$. Then the
morphism~$L^\ab_!(\Wh{A}{j})$ is a trivial cofibration in the projective
model structure on $\Ch(\Ab)$. 
\end{coro}
\begin{proof}
Since we proved in Proposition \ref{prop:colimiteHomotopiqueLibres}
that the functor $L^\ab_!\Whf{A}$ sends weak equivalences of
presheaves on $A$ to quasi-isomorphisms,
we can conclude using the proposition above.
\end{proof}

\begin{theorem}\label{thm:modelStructureAbPresheaves}
Let $A$ be a strong Whitehead and local test category. The category
$\prefab{A}$ can be given the structure of a model category where weak
equivalences are the elements of~$\Wab_A$ and fibrations are the morphisms of
abelian presheaves whose underlying morphism of presheaves is a
fibration in the Grothendieck-Cisinski model structure on $\pref{A}$.
\end{theorem}
\begin{proof}
We will use Theorem \ref{th:CransTransfert} taking for $G$ the functor
\[
\Whf{A}: \pref{A}\to\prefab{A} \pbox{.}
\]
Since $\prefab{A}$ is locally presentable, we only need to show that
if $J$ is a set of generating trivial cofibrations for $\pref{A}$,
then we have the inclusion $l(r(\Z J) \subset \U^{-1}\W_A$. Let $L$ be
a free integrator on~$A$. Since $A$ is a strong Whitehead category, we have
\[
\U^{-1}\W_A = (L_!^\ab)^{-1}\W_{\mathrm{qis}}
\]
and, therefore, we only need to show that the functor $L^\ab_!:
\prefab{A} \to \Ch(\Ab)$ sends elements of $l(r(\Z J))$ to
quasi-isomorphisms. Since the functor $L^\ab_!$ is a left adjoint, the
small object argument (see for example \cite[Corollary
10.5.22]{hirschhorn2003model}) implies that 
\[
L^\ab_!(l(r(\Z J))) \subset l(r(L^\ab_!\Z(J))) \pbox{.}
\]
Since we just showed that elements of $L^\ab_!(\Z J)$ are trivial
cofibrations in $\Ch(\Ab)$, and since trivial cofibrations are stable
under pushouts, transfinite compositions and retracts, we can
conclude that every element of $l(r(\Z J))$ is also in $\Wab_A$.
\end{proof}

\begin{remark}
The cofibrations in this model structure are generated by the image of
monomorphisms by the functor $\Whf{A}: \pref{A}\to\prefab{A}$. In
particular, every free abelian presheaf is a cofibrant object.
\end{remark}

\begin{prop}\label{prop:foncteurAspheriqueEqQuillenAb}
Let $A$ be a strong Whitehead and local test category, $B$ a local test
category and $u: A \to B$ an aspherical functor, Then:
\begin{enumerate}
\item $B$ is a strong Whitehead category, and can be endowed with the model
structure defined in Theorem \ref{thm:modelStructureAbPresheaves};
\item the functors
\[
  (u^*)^\ab: \prefab{B} \to \prefab{A}\mdvirg u_*^\ab: \prefab{A} \to
  \prefab{B}
\]
forms a Quillen equivalence for the model structure defined in Theorem
\ref{thm:modelStructureAbPresheaves}.
\end{enumerate}
\end{prop}
\begin{proof}
We have already proved the first point in
Proposition~\ref{prop:aspheriqueSourceWhitehead}. For the second point,
recall that cofibrations and trivial cofibrations in $\prefab{B}$
are generated by the image under $\Whf{B}: \pref{B} \to \prefab{B}$
of cofibrations and trivial cofibrations in the Grothendieck-Cisinski
structure  on~$\pref{B}$. To prove that $(u^*)^\ab$ is a left Quillen
functor, we only have to check that it sends morphisms of the form
$\Z_B^{(i)}$, where $i$ is a cofibration in $\pref{B}$, 
to cofibrations of $\prefab{A}$ (since we already know that it preserves weak
equivalences). This follows from the fact that
$\Whf{A}$ and $u^*$ are left Quillen functors, and by the
commutativity of diagram \ref{diag:invImageCommutationsOubli}.

Let's prove that $((u^*)^\ab, u^\ab_*)$ is a Quillen equivalence. Let $X$ be an
abelian presheaf on ${B}$ and~$Y$ a fibrant object of $\prefab{A}$.
Then $\U(Y)$ is a fibrant object of $\pref{A}$, and $\U(X)$ is a cofibrant
object of $\pref{B}$. Using Proposition
\ref{StructureCisinskiMorphAspheriquesEqQuillen}, we can then conclude
from the following chain of equivalences:
\begin{align*}
& f: (u^*)^\ab X \to Y \in \Wab_A \\
&\iff \U f: \U(u^*)^\ab X \to \U Y \in \W_A & (\text{strong Whitehead
condition for } A)
\\
&\iff \U f: u^*\U X \to \U Y \in \W_A 
&(\text{commutativity of \ref{diag:invImageCommutationsOubli}})
\\
&\iff (\U f)^\sharp: \U X \to u_*\U Y \in \W_B
&(\text{Proposition } \ref{StructureCisinskiMorphAspheriquesEqQuillen})
\\
&\iff \U(f^\sharp): \U X \to \U u_*^\ab Y \in \W_B
&(\text{commutativity of \ref{diag:invImageCommutationsOubli}})
\\
&\iff f^\sharp: X \to u_*^{\ab} Y \in \Wab_B &(\text{strong Whitehead
condition for }B & )
\end{align*}
where we used the symbol $(-)^\sharp$ to denote the transpose of morphisms.
\end{proof}

\begin{coro}\label{coro:fonctAspheriqueTestLocalWhiteheadPsTestHomEq}
Let $A$ be a strong Whitehead and local test category, $B$ a local test
category, and~$u: A \to B$ an aspherical functor. Then $A$ is a
homologically pseudo-test category (\ref{def:homPseudoTest}) if and only if the same holds for~$B$.
\end{coro}
\begin{proof}
By Proposition \ref{prop:aspheriqueSourceWhitehead}, since $u$ is
aspherical and $A$ is a strong Whitehead category, then $B$ is a strong Whitehead
category. Thus, we can endow $\prefab{B}$ with the model category
structure described in Theorem \ref{thm:modelStructureAbPresheaves}.
Since aspherical morphisms are $\Wab_\infty$-aspherical, the diagram
\[
\xymatrix{
\Hotab_B \ar[rr]^{\overbar{(u^*)^\ab}} \ar[rd]_{\Hf{B}} && \Hotab_{A}
\ar[ld]^{\Hf{A}} \\
& \Hotab 
} 
\]
is commutative up to a natural isomorphism, and Proposition
\ref{prop:foncteurAspheriqueEqQuillenAb} 
shows that the horizontal arrow is an equivalence of categories. We
can then conclude that $\Hf{A}$ is an equivalence of categories if and
only if $\Hf{B}$ is an equivalence of categories.
\end{proof}

\begin{example}
As a direct application of this corollary, the category $\Delta^n$ for $n>0$ is
a homologically pseudo-test category, as explained in example
\ref{ex:WhiteheadCategories}.
\end{example}

\begin{theorem}\label{thm:WhiteheadTestHomPseudoTest}
If $A$ is a strong Whitehead and test category, then $A$ is a
homologically pseudo-test category.
\end{theorem}
\begin{proof}
Let $A_0$ be a homologically pseudo-test, strong Whitehead and
strict test category (for example, we can take the category $\Delta$
of simplices). Let $B$ be the full subcategory of $\Cat$ whose objects
are the images by $i_A$ of objects of $A$ and images by $i_{A_0}$ of
objects of $A_0$. We then get the following commutative diagram
\[
\xymatrix{
& \Cat 
\\
A \ar[r]_{u} \ar[ru]^{i_A} &
B \ar@{^{(}->}[u]^{i} &
A_0 \ar[l]^{v} \ar[lu]_{i_{A_0}} 
} 
\]
in $\Cat$. We will show that $u$ and $v$ are aspherical functors, and
that $B$ is a strong Whitehead and local test category, before applying
Corollary \ref{coro:fonctAspheriqueTestLocalWhiteheadPsTestHomEq}.
Since $A_0$ is a local test category, the functor~$i_{A_0}$ is
aspherical by Proposition \ref{prop:testFaiblei_AAspherique}.
Moreover $i$ is full and faithful, so Proposition
\ref{lemmeFoncteursAspheriquesTriCommutatif} implies that $i$ and $v$ are
aspherical. Since $A$ is a local test category, the same argument shows that
$i_A$ and $u$ are also aspherical. Since $A_0$ is a strong Whitehead category
and $v$ is aspherical, Proposition \ref{prop:aspheriqueSourceWhitehead} implies
that $B$ is a strong Whitehead
category,

Let's show that $B$ is a local test category. 
Consider a small aspherical category $C$. Since $i$ is aspherical,
$i^*(C)$ is an aspherical presheaf on $B$ by Theorem
\ref{thm:EquivalencesFoncteursTestLocal}. But since $v$ is aspherical and $A_0$
is totally aspherical,
Proposition \ref{totAspheriqueMorphismeAspherique} implies that $B$ is also
totally aspherical. This implies by Proposition
\ref{prop:totalementAspheriqueEquivalences} that $i^*(C)$ is a locally
aspherical presheaf. By Theorem
\ref{thm:EquivalencesFoncteursTestLocal}, this implies that $B$ is a
local test category and that $i$ is a local test functor.
Finally, since $A$, $A_0$ and $B$ are strong Whitehead and local test
categories, and since there is a zig-zag of aspherical functors
\[
A \xrightarrow{u} B \xleftarrow{v} A_0 \pbox{,}
\]
we can conclude using Corollary
\ref{coro:fonctAspheriqueTestLocalWhiteheadPsTestHomEq} that $A$ is a
homologically pseudo-test category, since the same holds for
$A_0$.
\end{proof}

\begin{remark}
We showed in Proposition \ref{prop:intLibreMonos} that if $L$ is a free
integrator on a small category $A$, the functor
\[
  L_!^\ab: \prefab{A} \to \Ch(\Ab)
\]
sends monomorphisms of free abelian presheaves to monomorphisms with projective
cokernels. Moreover, any small category can be equipped a free integrator, as
explained in paragraph \ref{paragr:BousfieldKanIntegrator}. In particular, if
$A$ is a strong Whitehead and test category, we actually get a Quillen
equivalence between the model structure on $\prefab{A}$ given in Theorem
\ref{thm:modelStructureAbPresheaves} and the projective model
structure on~$\Ch(\Ab)$. 
\end{remark}

\section{\texorpdfstring{Joyal's category $\Theta$ is a homologically
pseudo-test category}%
{Joyal's category {Θ} is a homologically pseudo-test category}}
\label{sec:theta}

We first recall the construction of $\Theta$ using wreath products, due to
Berger \cite{berger2007iterated}, which makes it easy to define an
aspherical functor $\Delta^n \to \Theta_n$ for every integer $n \geq 0$,
allowing to apply directly~Corollary
\ref{coro:fonctAspheriqueTestLocalWhiteheadPsTestHomEq}.
For $\Theta$, we will use a variation on the notion of asphericity for functors
whose codomain is a presheaf category and show the analogue of Proposition
\ref{prop:aspheriqueSourceWhitehead} in this context. This will prove, using
previous results from Ara and Maltsiniotis \cite{ara2022compnerfs}, that
$\Theta$ is a strong Whitehead category.

\begin{paragr}
Let $A$ be a small category. We denote by $\Delta\wr A$ the small
category whose objects are pairs $[\Delta_n;(a_i)_{1\leq i \leq n}]$ where
$n$ is a non-negative integer and $(a_i)_{1 \leq i \leq n}$ is a
family of object of $A$, and whose morphisms 
\[
  [\Delta_n;(a_i)_{1 \leq i \leq n}] \to [\Delta_m;(a'_j)_{1 \leq j
  \leq m}]
\]
are pairs $[\varphi, \mathbf{f}]$ where $\varphi: \Delta_n \to
\Delta_m$ is a morphism of $\Delta$ and 
\[
    \mathbf{f}=(f_{ji}:a_i \to a'_j)
    _{
      1 \leq i \leq n,\,\varphi(i-1)<j\leq\varphi(i)
    }
\]
is a family of morphisms of $A$. Given two composable morphisms
\[
  [\Delta_{n}  ;( a_i)_{1 \leq i \leq n} ] \xrightarrow{[\varphi, \mathbf{f}]}
  [\Delta_{n'} ;( a'_{i'})_{1 \leq i' \leq n'}] \xrightarrow{[\varphi', \mathbf{f'}]}
  [\Delta_{n''};( a''_{i''})_{1 \leq i'' \leq n''}] \mdvirg
\]
their composite is defined by the pair $[\varphi'',\mathbf{f''}]$
with $\varphi''=\varphi'\varphi$ and
\[
\mathbf{f''}=(f''_{i''i})_{
    1 \leq i \leq n~,~
    \varphi''(i-1) < i'' \leq \varphi''(i) } \mdvirg
    f''_{i''i} = f'_{i''i'}f_{i'i}
\]
where $i'$ is the only integer such that $\varphi(i-1) < i' \leq
\varphi(i)$ and $ \varphi'{(i'-1)} < i'' \leq \varphi'(i')$. 
\end{paragr}

\begin{paragr}
\label{def:mu_A}
Given a functor $F: A \to B$ between small categories, we denote by 
\[
\Delta\wr F: \Delta\wr A \to \Delta\wr B  
\]
the functor sending objects $[\Delta_n;(a_i)_i]$ of
$\Delta \wr A$ to $ [\Delta_n;(F(a_i))_{i}]$, and
morphisms $[\varphi, \mathbf{f}=(f_{ji})_{i,j}]$ to 
\[
(\Delta\wr F)[\varphi,\mathbf{f}] = [\varphi, F(\mathbf{f}) =
\left(F(f_{ji})\right)_{i,j}] \pbox{.}
\]

Note that if $F$ is a faithful (resp.~fully faithful) functor, then
the same is true for the functor $\Delta\wr F$. In particular, if $a$
is an object of $A$, we get a faithful functor

\[
I_a: \Delta \wr e \simeq \Delta \xrightarrow{\Delta \wr a} \Delta \wr
A \pbox{,}
\]
where we also denoted by $a$ the corresponding functor from the
terminal category $e$ to $A$. This functor is fully faithful if $a$
doesn't admit non-trivial endomorphisms. 
This construction is also functorial in $a$, and we get a functor
denoted 
\[
\mu_A: \Delta \times A \to \Delta \wr A \mdvirg (\Delta_n, a) \mapsto
I_a(\Delta_n) \pbox{.}
\]
\end{paragr}

\begin{paragr}

\label{defThetaCouronne} For any integer $n\geq 0$, we denote by
$\Theta_n$ the category defined inductively in the following way:

\begin{itemize}
\item $\Theta_0 = e$ is the final small category; 
\item $\Theta_{n+1} = \Delta \wr \Theta_n$ for $n\geq 0$.
\end{itemize}
Notice that since $\Delta_0$ is the final object of $\Delta$, we get a
fully faithful functor $\Theta_n \hookrightarrow \Theta_{n+1}$ for any
integer~$n \geq 0$. We denote by $\Theta$ the colimit of the diagram
\[
\Theta_0 \hookrightarrow \Theta_1 \hookrightarrow \Theta_2
\hookrightarrow \cdots 
\]
in $\Cat$.
\end{paragr}

\begin{remark}\label{remark:thetaAltDef}
This category is equivalent to the category introduced by Joyal in
\cite{joyal1997disks}. A description of $\Theta$ as a full subcategory
of the category $\omega\Cat$ was conjectured in
\cite{batanin2000multitude} and proven independently in
\cite{makkai2001duality} and~\cite{berger2007iterated}. 
\end{remark}

\begin{theorem}[Cisinski-Maltsiniotis]\label{thm:ThetaTestStricte}
For any integer $n\geq1$, the category $\Theta_n$ is a strict test
category. Moreover, the category $\Theta$ is also a strict test
category.
\end{theorem}
\begin{proof}
See \cite{cisinski2011theta}.
\end{proof}

\begin{paragr}\label{def:m_n} 
For any integer $n\geq0$, we define a functor
\begin{align*}
m_n: \Delta^n &\to \Theta_n 
\end{align*}
inductively in the following way:
\begin{itemize}
\item $m_0$ is the identity on $e=\Delta^0=\Theta_0$;
\item $m_{n+1}$ is the composite
\[
\Delta^{n+1} = \Delta \times \Delta^n \xrightarrow{\Delta\times m_n}
\Delta\times\Theta_n \xrightarrow{\mu_{\Theta_n}} \Theta_{n+1} \pbox{.}
\]
\end{itemize}
\end{paragr}

\begin{prop}[Ara-Maltsiniotis]
\label{prop:m_nAspherical}
For any integer $n\geq 0$, 
\[
m_n: \Delta^n \to \Theta_n 
\]
is an aspherical functor.
\end{prop}
\begin{proof}
See \cite[Corollary 5.5]{ara2022compnerfs}. 
\end{proof}

\begin{coro}
For any integer $n\geq 1$, $\Theta_n$ is a homologically pseudo-test
category.
\end{coro}
\begin{proof}
Since $\Delta$ is totally aspherical, the diagonal functor $\Delta \to
\Delta^n$ is aspherical. The composite functor $\Delta \to \Delta^n \to
\Theta_n$ is then also aspherical. Moreover, $\Delta$ is a strong Whitehead
category, and Proposition \ref{prop:aspheriqueSourceWhitehead} implies
that $\Theta_n$ is a strong Whitehead category. Since, by Proposition
\ref{thm:ThetaTestStricte}, $\Theta_n$ is also a test category, this
is a particular case of Theorem \ref{thm:WhiteheadTestHomPseudoTest}.
\end{proof}

We will now prove that $\Theta$ is also a homologically pseudo-test
category. For this, we will use a variation on the notion of aspherical
functors.

\begin{paragr}
\label{par:defFoncteurAspheriqueÂ} Let $\W\subset\Arr(\Cat)$ be a basic
localizer. If $j: B \to \pref{A}$ is a functor, we get a functor 
\[
j^*: \pref{B}\to \pref{A} \mdvirg X \mapsto 
( b \mapsto \Hom_{\pref{A}}(j(b),X) \pbox{.}
\]
We say that
$j$ is \emph{aspherical} if the following conditions are satisfied:
\begin{enumerate}
\item for every object $b$ of $B$, the presheaf $j(b)$ on $A$ is
aspherical;
\item for every object $b$ of $B$, the presheaf $j^*j(b)$ on $B$ is
aspherical;
\item for every object $a$ of $A$, the presheaf $j^*(a)$ on $B$ is
aspherical.

\end{enumerate}
\end{paragr}

\begin{paragr}
\label{paragr:diagFoncteurAspheriqueÂ} 
Let $j: B \to \pref{A}$ be a functor. Consider the full subcategory
$C$ of $\pref{A}$ whose objects are representables presheaves, and
images by $j$ of objects of $B$. By construction, the diagram

\begin{equation}\label{diagMorphismesAspheriquesStructureÂ}
\xymatrix{
& \pref{A} \\
B \ar[ru]^{j} \ar[r]_{j'} & C \ar@{^{(}->}[u]^{i} & A \ar@{_{(}->}[lu]_{h} \ar[l]^{h'} 
\pbox{,}
} 
\end{equation}
where we denoted by $h$ the Yoneda embedding, is commutative. This
implies that the diagram
\begin{equation}\label{diagMorphismesAspheriquesStructureÂ2}
\xymatrix{
& \pref{A} \ar[ld]_{j^*} \ar[d]^{i^*} \ar[rd]^{\id} \\
\pref{B} & \pref{C} \ar[l]^{j'^*} \ar[r]_{h'^*} & \pref{A}
} 
\end{equation}
is also commutative. 
\end{paragr}

\begin{prop}\label{propFoncteurAspheriqueZigZagMorphAsph}
In diagram \ref{diagMorphismesAspheriquesStructureÂ}, if $j$ is an
aspherical functor, then $j'$ and $h'$ are aspherical.
\end{prop}
\begin{proof}
Let $c$ be an object of $C$. Recall from Proposition
\ref{prop:MorphismesAspheriquesLocFond} that we only need to show that
$j'^*(c)$ and $h'^*(c)$ are aspherical presheaves. Since $i$ is fully
faithful, there is a natural isomorphism $i^*i(c) \simeq c$ in
$\pref{C}$, which induces a natural isomorphism
\[
  h'^*(c) \simeq h'^*i^*i(c) \simeq i(c) \pbox{.}
\]
We distinguish two cases: if $c$ is an object of $A$, then $i(c)$ is a
representable presheaf, which implies that it is aspherical. If $c$ is
a $j'(b)$ for an object $b$ of $B$, then $i(c)=j(b)$ is an aspherical
functor by the first condition of paragraph \ref{par:defFoncteurAspheriqueÂ}. In
both cases, $h'^*(c)$ is an aspherical presheaf on $A$, which proves
that $h'$ is aspherical.

We proceed in the same way to show that $j$ is aspherical: we have a
natural isomorphism
\[
j'^*(c) \simeq j'^*i^*i(c) \simeq j^*i(c)
\]
and we distinguish two cases: either $c$ is an object of $A$, in which
case $j^*i(c)$ is aspherical by condition $3$ of paragraph
\ref{par:defFoncteurAspheriqueÂ}, or $c$ is a $j'(b)$ for an
object $b$ of $B$, in which case $j^*ij'(b)=j^*j(b)$ is aspherical by
condition $2$. We have proved that $j'$ is also an aspherical functor.
\end{proof}

\begin{prop}\label{prop:FoncteurAspheriqueÂSourceWhitehead}
Let $A$ be a small category and $B$ a strong Whitehead category. If $j: B
\to \pref{A}$ is an aspherical functor, then $A$ is a strong Whitehead
category.
\end{prop}
\begin{proof}
We keep the notation of paragraph \ref{paragr:diagFoncteurAspheriqueÂ}. 
Since $i^*$ and ${h'^*}$ are right adjoint functors, they preserve
abelian group objects and induce functors 
\[
(i^*)^{\ab}: \prefab{A}\to\prefab{C} \mdvirg (h'^*)^{\ab}: \prefab{C} \to
\prefab{A}
\]
such that, in diagram
\[
\xymatrix{
\pref{A} \ar[d]_{\Whf{A}} \ar[r]^{i^*}  &
\pref{C} \ar[d]^{\Whf{C}} \ar[r]^{h'^*} &
\pref{A} \ar[d]^{\Whf{A}} \\
\prefab{A} \ar[d]_{\U_{A}} \ar[r]_{(i^*)^{\ab}}
&
\prefab{C} \ar[d]^{\U_{C}} \ar[r]_{(h'^*)^{\ab}} &
\prefab{A} \ar[d]^{\U_{A}} \\
\pref{A}  \ar[r]_{i^*} &
\pref{C}  \ar[r]_{h'^*} &
\pref{A} \pbox{,}
} 
\]
every square except the top-left one are commutative.
Since $h'^*i^* = \id$, the exterior square is also commutative. Notice
also that we have $(h'^*)^{\ab}(i^*)^{\ab} = \id$.

By Proposition \ref{prop:aspheriqueSourceWhitehead}, we know that $C$
is a strong Whitehead category, since we proved that $j'$ is aspherical.
Moreover, $h'$ is also an aspherical morphism. We can then conclude by
the following chain of equivalences, for any morphism $f$
of $\pref{A}$:
\begin{align*}
f \in \Wab_A &\iff (h'^*)^{\ab}(i^*)^{\ab}f \in \Wab_A 
\\&\iff (i^*)^{\ab}f \in \Wab_C  &\text{ (asphericity for $h'$)}
\\&\iff \U_C (i^*)^{\ab}f \in \W_C &\text{ (strong Whitehead condition for $C$)}
\\&\iff i^* \U_Af \in \W_C &\text{ (commutativity of bottom-left square)}
\\&\iff h'^*i^*\U_Af \in \W_A     &\text{ (asphericity for $h'$)}
\\&\iff \U_A(f)\in\W_A \pbox{.}  
&& \qedhere
\end{align*}
\end{proof}

We will now show that $\Theta$ is a strong Whitehead category. For this, we just
point out the existence of an aspherical functor $\Delta \to \pref{\Theta}$ and
refer to the literature for an explicit description, as our purpose is only to
apply proposition \ref{prop:FoncteurAspheriqueÂSourceWhitehead}.

\begin{paragr}
\label{paragrCompNerfs}

Denote by $\omega\Cat$ the category of strict~$\omega$-categories and strict
$\omega$-functors. There is an inclusion~$\Theta \hookrightarrow \wcat$ (see
\cite{makkai2001duality}, \cite{berger2002cellular}) that allows to define a fully faithful functor called
the \emph{cellular nerve}
\[
\nerf_{\Theta}: \wcat \to \pref{\Theta} \pbox{.}
\]
Moreover, Street introduced in \cite{street1987algebra} a functor
$\orient: \Delta \to \wcat$, which induces a functor 
\[
\nerf: \wcat \to \pref{\Delta} 
\]
called the \emph{Street nerve}.
In \cite{ara2022compnerfs}, Ara and Maltsiniotis prove that these two
nerve functors are equivalent, meaning that the square
\[
\xymatrix{
\wcat \ar[r]^{\nerf_\Theta} \ar[d]_{\nerf} & \pref{\Theta}
\ar[d]^{i_{\Theta}} \\
\pref{\Delta} \ar[r]_{i_{\Delta}} & \Hot
} 
\]
is commutative up to a natural isomorphism. One of the steps of their
proof is the following:
\end{paragr}

\begin{theorem}[Ara-Maltsiniotis]
The functor
\[
j: \Delta \xrightarrow{\orient} \wcat \xrightarrow{N_\Theta} \pref{\Theta} 
\]
is an aspherical functor, within the meaning of paragraph
\ref{par:defFoncteurAspheriqueÂ}.
\end{theorem}
\begin{proof}
See \cite[Theorem~5.8]{ara2022compnerfs}. 
\end{proof}

\begin{theorem}
The category $\Theta$ is a strong Whitehead and homologically pseudo-test
category. 
\end{theorem}
\begin{proof}
By Proposition \ref{prop:FoncteurAspheriqueÂSourceWhitehead}, the asphericity of
$j$ implies that $\Theta$ is a strong Whitehead category. Since it is also a
test category, we can conclude with Theorem
\ref{thm:WhiteheadTestHomPseudoTest}.
\end{proof}

\bibliographystyle{amsplain}
\bibliography{biblio}

\end{document}